\documentclass[11pt]{article}

\newif\ifarxiv
\arxivtrue   

\newif\ifsupplement
\supplementtrue   

\arxivfalse
\supplementfalse

\arxivtrue 
\supplementfalse

\arxivfalse
\supplementtrue

\usepackage{setspace} 
\usepackage{tikz}
\usepackage{pgfplots}

\usepackage{fullpage} 
\usepackage{multicol}
\usepackage{stmaryrd}
\usepackage{amsmath}
\usepackage{amsthm}
\usepackage{hyperref}
\usepackage{amssymb}
\usepackage{makecell}

\newtheorem{theorem}{Theorem}[section]

\newtheorem{lemma}[theorem]{Lemma}
\newtheorem{conjecture}{Conjecture}
\newtheorem{definition}{Definition}
\newtheorem{proposition}[theorem]{Proposition}

\newtheorem{claim}[theorem]{Claim}

\newtheorem{question}{Question}

\usetikzlibrary { decorations.pathmorphing, decorations.pathreplacing, decorations.shapes, }

\tikzset{vtx/.style={inner sep=1.7pt, outer sep=0pt, circle, fill=black,draw}}

\newcommand{\vc}[1]{\ensuremath{\vcenter{\hbox{#1}}}}
\tikzset{flag_pic/.style={scale=1}}  
\tikzset{unlabeled_vertex/.style={inner sep=1.7pt, outer sep=0pt, circle, fill}}
\tikzset{labeled_vertex/.style={inner sep=3pt, outer sep=0pt, rectangle, fill=white, draw=black}}
\tikzset{edge_color0/.style={color=black,line width=1.2pt,opacity=0.5,dashed}}
\tikzset{edge_color1/.style={color=red,  line width=1.2pt,opacity=0}} 
\tikzset{edge_color2/.style={color=black, line width=1.2pt,opacity=1}}
\tikzset{edge_color3/.style={color=blue,line width=1.5pt,densely dotted}}
\tikzset{edge_color_blue/.style={color=blue,line width=1.5pt,densely dotted}}
\tikzset{edge_color4/.style={color=red, line width=1pt}}
\tikzset{edge_color_red/.style={color=red, line width=1pt}}
\tikzset{edge_color5/.style={color=magenta,  line width=1.4pt,dash pattern = {on 4pt off 2pt on 1pt off 2pt}}}
\tikzset{edge_color6/.style={color=gray, line width=1.4pt,densely dashed}}
\tikzset{edge_color7/.style={color=cyan, line width=1.4pt,dash pattern = {on 2pt off 3pt}}}
\tikzset{edge_color8/.style={color=gray, line width=1.2pt}}
\tikzset{edge_color9/.style={color=gray, dotted, line width=1.2pt}}
\tikzset{edge_color10/.style={color=gray, dashed, line width=1.2pt}}
\tikzset{edge_color11/.style={color=pink, dashed, line width=1.2pt}}
\tikzset{edge_colorroot/.style={color=red, line width=1.7pt}}
\tikzset{edge_thin/.style={color=black}}
\tikzset{edge_hidden/.style={color=black,dotted,opacity=0}}
\tikzset{vertex_color0/.style={inner sep=1.7pt, outer sep=0pt, draw, circle, fill=black}}
\tikzset{vertex_color1/.style={inner sep=1.7pt, outer sep=0pt, draw, circle, fill=red}}
\tikzset{vertex_color2/.style={inner sep=1.7pt, outer sep=0pt, draw, circle, fill=blue}}
\tikzset{vertex_color3/.style={inner sep=1.7pt, outer sep=0pt, draw, circle, fill=green!80!black}}
\tikzset{vertex_color4/.style={inner sep=1.7pt, outer sep=0pt, draw, circle, fill=pink}}
\tikzset{vertex_color5/.style={inner sep=1.7pt, outer sep=0pt, draw, circle, fill=gray,label=below:{$5$}}}
\tikzset{vertex_color6/.style={inner sep=1.7pt, outer sep=0pt, draw, circle, fill=gray,label=below:{$6$}}}
\tikzset{vertex_color7/.style={inner sep=1.7pt, outer sep=0pt, draw, circle, fill=gray,label=below:{$7$}}}
\tikzset{vertex_color8/.style={inner sep=1.7pt, outer sep=0pt, draw, circle, fill=gray,label=below:{$8$}}}
\tikzset{vertex_color9/.style={inner sep=1.7pt, outer sep=0pt, draw, circle, fill=gray,label=below:{$9$}}}
\tikzset{vertex_color10/.style={inner sep=1.7pt, outer sep=0pt, draw, circle, fill=gray,label=below:{$10$}}}
\tikzset{vertex_color11/.style={inner sep=1.7pt, outer sep=0pt, draw, circle, fill=gray,label=below:{$11$}}}
\tikzset{vertex_color12/.style={inner sep=1.7pt, outer sep=0pt, draw, circle, fill=gray,label=below:{$12$}}}
\tikzset{vertex_color13/.style={inner sep=1.7pt, outer sep=0pt, draw, circle, fill=gray,label=below:{$13$}}}
\tikzset{vertex_color14/.style={inner sep=1.7pt, outer sep=0pt, draw, circle, fill=gray,label=below:{$14$}}}
\tikzset{labeled_vertex_color0/.style={inner sep=2.2pt, outer sep=0pt, draw, rectangle, fill=black}}
\tikzset{labeled_vertex_color1/.style={inner sep=2.2pt, outer sep=0pt, draw, rectangle, fill=red}}
\tikzset{labeled_vertex_color2/.style={inner sep=2.2pt, outer sep=0pt, draw, rectangle, fill=blue}}
\tikzset{labeled_vertex_color3/.style={inner sep=2.2pt, outer sep=0pt, draw, rectangle, fill=green}}
\tikzset{labeled_vertex_color4/.style={inner sep=2.2pt, outer sep=0pt, draw, rectangle, fill=pink}}
\tikzset{labeled_vertex_color5/.style={inner sep=2.2pt, outer sep=0pt, draw, rectangle, fill=gray,label=below:{$5$}}}
\tikzset{labeled_vertex_color6/.style={inner sep=2.2pt, outer sep=0pt, draw, rectangle, fill=gray,label=below:{$6$}}}
\tikzset{labeled_vertex_color7/.style={inner sep=2.2pt, outer sep=0pt, draw, rectangle, fill=gray,label=below:{$7$}}}
\tikzset{labeled_vertex_color8/.style={inner sep=2.2pt, outer sep=0pt, draw, rectangle, fill=gray,label=below:{$8$}}}
\tikzset{labeled_vertex_color9/.style={inner sep=2.2pt, outer sep=0pt, draw, rectangle, fill=gray,label=below:{$9$}}}
\tikzset{text_color0/.style={color=black}}
\tikzset{text_color1/.style={color=red}}
\tikzset{text_color2/.style={color=blue}}
\tikzset{text_color3/.style={color=green!70!black}}
\tikzset{text_color4/.style={color=orange}}
\tikzset{text_color5/.style={color=gray}}

\def\outercycle#1#2{ 
\pgfmathtruncatemacro{\plusone}{#1+1} 
\pgfmathtruncatemacro{\zeroshift}{270 - (#2-1)*360/#1/2 } 
\draw  \foreach \x in {0,1,...,#1}{(\zeroshift+\x*360/#1:1) coordinate(x\x)};}

\def\labelvertex#1{\pgfmathtruncatemacro{\vertexlabel}{#1+1 } \draw (x#1) node{\color{black}\tiny\vertexlabel}; }

\usepackage{ifthen}

\tikzset{vertex_u/.style={unlabeled_vertex}} 
\tikzset{vertex_l/.style={labeled_vertex}}

\newcommand{\Fuu}[1]{
\,\vc{\begin{tikzpicture}[scale=0.3]\outercycle{2}{1}
\draw[edge_color#1] (x0)--(x1);  
\draw (x0) node[unlabeled_vertex]{};\draw (x1) node[unlabeled_vertex]{};
\end{tikzpicture}}
\,
}

\newcommand{\Fllu}[3]{
\vc{\begin{tikzpicture}[scale=0.4]\outercycle{3}{2}
\draw[edge_color#1] (x0)--(x1);\draw[edge_color#2] (x0)--(x2);  \draw[edge_color#3] (x1)--(x2);    
\draw (x0) node[labeled_vertex]{};\draw (x1) node[labeled_vertex]{};\draw (x2) node[unlabeled_vertex]{};
\labelvertex0
\labelvertex1
\end{tikzpicture}}}

\newcommand{\FfourEdges}[6]{
\draw[edge_color#1] (x0)--(x1);\draw[edge_color#2] (x0)--(x2);\draw[edge_color#3] (x0)--(x3);  \draw[edge_color#4] (x1)--(x2);\draw[edge_color#5] (x1)--(x3);  \draw[edge_color#6] (x2)--(x3);
}
\newcommand{\Ffour}[5]{
\vc{\begin{tikzpicture}[scale=0.4]\outercycle{4}{2}
\FfourEdges#5
\draw (x0) node[vertex_#1]{};\draw (x1) node[vertex_#2]{};\draw (x2) node[vertex_#3]{};\draw (x3) node[vertex_#4]{};
\ifthenelse{\equal{#1}{l}}{\labelvertex{0}}{}%
\ifthenelse{\equal{#2}{l}}{\labelvertex{1}}{}%
\ifthenelse{\equal{#3}{l}}{\labelvertex{2}}{}%
\ifthenelse{\equal{#4}{l}}{\labelvertex{3}}{}%
\end{tikzpicture}}
}

\newcommand{\Fuuuu}[6]{\Ffour{u}{u}{u}{u}{#1#2#3#4#5#6}}

\usepackage{listofitems}
\tikzset{vertex_u/.style={unlabeled_vertex}}
\tikzset{vertex_l/.style={labeled_vertex}}
\newcommand{\Fnv}[2]{ 
\ifnum#2<#1 \draw (x#2) node[vertex_l]{}; \labelvertex{#2}  
\else  \draw (x#2) node[vertex_u]{}; \fi 
}
\newcommand{\Fne}[3]{
\draw[edge_color#3] (x#1)--(x#2); 
}
\newcounter{Fneid} 
\newcommand{\Fe}[3]{
\ifnum#1=1
  \vc{\begin{tikzpicture}[scale=0.4]\outercycle{1}{2}
  \Fnv{#2}{0}
  \end{tikzpicture}}
\else
\setsepchar{ }
\readlist\elabel{#3}
\pgfmathtruncatemacro{\vertexloop}{#1-1}
\pgfmathtruncatemacro{\vertexloopi}{#1-2}
\pgfmathtruncatemacro{\expectededges}{#1*(#1-1)/2}
\ifnum\elabellen=\expectededges%
\def\cycleshift{2}
\ifnum#1=2\def\cycleshift{1}\fi
\ifnum#1=3\ifnum#2=1\def\cycleshift{1}\fi\fi
\def\Fnscale{0.4}\ifnum#1>4\def\Fnscale{0.45}\fi\ifnum#1>5\def\Fnscale{0.55}\fi\ifnum#1>7\def\Fnscale{0.65}\fi
\vc{\begin{tikzpicture}[scale=\Fnscale]
          \outercycle{#1}{\cycleshift}
          \setcounter{Fneid}{1}         
          \foreach\i in {0,...,\vertexloopi}{   
          \pgfmathtruncatemacro{\jfrom}{\i+1}
          \foreach\j in {\jfrom,...,\vertexloop}{
          \edef\eID{\arabic{Fneid}}
          \edef\eij{\elabel[\eID]}         
            \Fne\i\j{\eij}              
	    \stepcounter{Fneid}
          }}
          \foreach\i in {0,...,\vertexloop}{\Fnv{#2}{\i}  
          }
          \end{tikzpicture}}%
\else
   #1 vertices need \expectededges{} edges but got \elabellen edges.
\fi
\fi
}

\usepackage{listofitems}
\tikzset{vertex_u/.style={unlabeled_vertex}}
\tikzset{vertex_l/.style={labeled_vertex}}
\newcommand{\FnvD}[3]{ 
\ifnum#2<#1 \draw (x#2) node[labeled_vertex_color#3]{}; \labelvertex{#2}  
\else  \draw (x#2) node[vertex_color#3]{}; \fi 
}

\newcounter{Fvneid} 
\newcounter{Fvnvid} 

\newcommand{\Fve}[4]{
\ifnum#1=1
  \vc{\begin{tikzpicture}[scale=0.4]\outercycle{1}{2}
  \FnvD{#2}{0}{#3}
  \end{tikzpicture}}
\else
\setsepchar{ }
\readlist\vlabel{#3}
\readlist\elabel{#4}
\pgfmathtruncatemacro{\vertexloop}{#1-1}
\pgfmathtruncatemacro{\vertexloopi}{#1-2}
\pgfmathtruncatemacro{\expectededges}{#1*(#1-1)/2}
\ifnum\vlabellen=#1%
\ifnum\elabellen=\expectededges%
\def\cycleshift{2}
\ifnum#1=2\def\cycleshift{1}\fi
\ifnum#1=3\ifnum#2=1\def\cycleshift{1}\fi\fi
\def\Fnscale{0.4}\ifnum#1>4\def\Fnscale{0.45}\fi\ifnum#1>5\def\Fnscale{0.55}\fi\ifnum#1>7\def\Fnscale{0.65}\fi
\vc{\begin{tikzpicture}[scale=\Fnscale]
          \outercycle{#1}{\cycleshift}
          \setcounter{Fvneid}{1}         
          \foreach\i in {0,...,\vertexloopi}{   
          \pgfmathtruncatemacro{\jfrom}{\i+1}
          \foreach\j in {\jfrom,...,\vertexloop}{
          \edef\eID{\arabic{Fvneid}}%
          \edef\eij{\elabel[\eID]} 
            \Fne\i\j{\eij}
	    \stepcounter{Fvneid}
          }}
          \foreach\i in {0,...,\vertexloop}{%
          \setcounter{Fvnvid}{\i}
          \stepcounter{Fvnvid}
          \edef\vID{\arabic{Fvnvid}}%
          \edef\vlabeli{\vlabel[\vID]}%
          \FnvD{#2}{\i}{\vlabeli} 
          }%
          \end{tikzpicture}}%
\else
   \text{ #1 vertices need \expectededges{} edges but got \elabellen{} edges. }
\fi
\else
 \text{ #1 vertices need #1 vertex labels but got \vlabellen{} labels. }
\fi
\fi
}

\tikzset{vertex/.style={inner sep=1.7pt, outer sep=0pt, circle, draw=black, fill=black}} 

\title{Semi-Inducibility of some small graphs}

\author{
J\'ozsef Balogh\thanks{Department of Mathematics, University of Illinois Urbana-Champaign, Urbana, IL, USA, and Extremal Combinatorics and Probability Group (ECOPRO), Institute for Basic Science (IBS), Daejeon, South Korea. Email: \texttt{jobal@illinois.edu}. Supported by NSF grants RTG DMS-1937241, FRG DMS-2152488, the Arnold O. Beckman Research Award (UIUC Campus Research Board RB 24012), the Simons Fellowship and the Institute for Basic Science (IBS-R029-C4).}
\and
Bernard Lidick\'{y}\thanks{Department of Mathematics, Iowa State University, Ames, IA. E-mail: {\tt lidicky@iastate.edu}. Research of this author is supported in part by NSF FRG DMS-2152490, the Simons Foundation TSM-00013439 and Scott Hanna Professorship.}
\and
Dhruv Mubayi\thanks{
Department of Mathematics, Statistics and Computer Science, University of Illinois, Chicago, IL 60607. E-mail: {\tt mubayi@uic.edu}. Research partially supported by NSF Award DMS-2153576.
}
\and
Florian Pfender\thanks{Department of Mathematical and Statistical Sciences, University of Colorado Denver, Denver, CO. E-mail:
{\tt Florian.Pfender@ucdenver.edu}. Research of this author is supported in part by NSF grant DMS-2152498.}
\and
Jan Volec\thanks{
Department of Theoretical Computer Science, Faculty of Information Technology,
Czech Technical University in Prague, Th\'akurova 9, Prague, 160 00, Czech Republic.
E-mail: {\tt jan@ucw.cz}.
Research partially supported by the grant 23-06815M of the Grant Agency of the Czech
Republic.
}
}

\date{\today}

\begin{document}

\maketitle

\begin{abstract}    
Let $H$ be a fixed graph whose edges are colored red and blue and let $\beta \in [0,1]$. Let $I(H, \beta)$ be the (asymptotically normalized) 
maximum number of copies of $H$ in a large red/blue edge-colored complete graph $G$, where the density of red edges in $G$ is $\beta$. This refines  the  problem of determining the semi-inducibility of $H$, which is itself a generalization of the classical question of determining the inducibility of $H$. The function $I(H, \beta)$ for $\beta \in [0,1]$ was not known for any graph $H$ on more than three vertices, except when $H$ is a monochromatic clique (Kruskal-Katona) or a monochromatic star (Reiher-Wagner). We obtain  sharp results for some four and five vertex graphs, addressing several recent questions posed by various authors. We also obtain some general results for trees and stars.  Many open problems remain.
\end{abstract}


\section{Introduction} 
Counting subgraphs inside a host graph is a fundamental problem in extremal combinatorics. A classical question is the following: given a fixed graph \(F\), what is the maximum number of induced copies of \(F\) in an \(n\)-vertex host graph? This is the inducibility problem introduced by 
Pippenger and Golumbic~\cite{Pippenger1975}. Here we focus on a red-blue ($2$-edge-colored) refinement of this classical problem.  Let \(H\) be a graph whose edges are colored red or blue, and let \(G\) be a red-blue colored  complete graph on $n$ vertices.  

\begin{definition} \label{defn:HG}
Write
$(\# H, G)$
for the number  of injections $f:V(H) \rightarrow V(G)$ such that $uv \in E(H)$ is red if only if $f(u)f(v) \in E(G)$ is red. Furthermore, if $H$ has $h$ vertices, 
\[
\max(H, n) \;:=\; \max_{\substack{G\; \text{a red-blue }K_n}} (\# H, G) \quad \quad  and \quad  \quad 
\max(H) \;:=\; \lim_{n\to\infty} \frac{\max(H,n)}{n^{h}}\,. 
\]  
\end{definition}
To clarify the normalizing factor in Definition~\ref{defn:HG},  if $H$ is a $K_3$ with all three edges red, then $(\#H, G)$ is six times the number of red $K_3$ in $G$. 
For $\max(H)$, we chose to use the normalizing factor $n^h$ instead of  $(n)_h$ in order to be consistent with the definition of $\max(H)$ defined earlier in~\cite{basit2025semiinducibilityproblem}. This will not matter, since our results and questions are asymptotic in nature.

The problem of determining $\max(H, n)$ and $\max(H)$ is referred to as the {\it semi-inducibility} problem.
  In the classical inducibility setting one demands that, on a chosen set of \(|V(H)|\) vertices of $G$, specified edges are present and all other edges are absent.  In contrast, in the semi-inducibility setting some edges of \(H\) must appear in the copy (say the red‐ones) and some edges must be absent (the blue-ones), but all other pairs  are allowed to be either present or absent: thus one allows a “semi-induced” structure.   

The semi-inducibility problem was introduced by Basit,  Granet,  Horsley,  K\"undgen, and Staden~\cite{basit2025semiinducibilityproblem}, where they also determined $\max(H)$ for alternating colored paths and walks and also for several graphs $H$ on four vertices.
Subsequently,  $\max(H)$ was determined for the alternating colored six cycle  by Chen and Noel~\cite{chen2025} 
and 
for several more four vertex graphs by Bodn\'ar and Pikhurko~\cite{bodnar2025}. Both~\cite{chen2025} and~\cite{bodnar2025} almost exclusively used  Flag Algebras.
The specific case of the alternating path was also investigated by Chen, Clemen and Noel~\cite{chen2025maximizingalternatingpathsentropy} using entropy.

Our goal is to initiate a systematic attack on a finer version of this problem: we determine or bound \(\max(H,n)\) (and hence \(\max(H)\)) for various small red-blue graphs \(H\), when the underlying graph $G$ has a given density $\beta$. More precisely, given $\beta \in [0,1]$ and red/blue graph $H$ with $h$ vertices, we consider the parameter $I(H, \beta)$ formally defined as follows.

\begin{definition} \label{defn:Hgood}
A sequence of red/blue colored cliques $(G_n)=(G_n)_{n=1}^{\infty}$ has density $\beta$ if $|V(G_n)|=n$ and the red density of edges in $G_n$ tends to $\beta$. The sequence $(G_n)$ is $H$-good if $\rho(H, (G_n)):= \lim_{n \rightarrow \infty} (\# H, G_n)/n^h$ exists. 
Then
$$I(H, \beta) = \sup \, \rho(H, (G_n)),$$
where the supremum is taken over all $H$-good sequences with density $\beta$.
\end{definition}
Clearly, $\max(H)= \sup_{\beta}I(H, \beta)$ over $\beta \in [0,1]$, but determining the function $I(H, \beta)$ for all $\beta$ is more challenging  and gives a more complete picture of the number of copies of $H$ that can appear in a large graph $G$. 
In order to standardize our notation,  if a pair is colored red, then we call it an edge and if it is colored blue then  we call it a non-edge. This allows us to dispense with red/blue colorings and instead we just consider graphs with certain non-edges specified.
In figures, we either use colors red and dotted blue\footnote{We also refer to the blue edges as non-edges.} to indicate that red and blue edges are fixed and the colors of the remaining non-edges are not specified or we use black and non-edges to describe induced graphs where all pairs are specified.

The function $I(H, \beta)$ is known for all monochromatic cliques and monochromatic stars. The former follows from the Kruskal-Katona theorem and the latter is a result of Reiher and Wagner~\cite{ReiherStar}. In addition, a more recent observation of Liu, Mubayi and Reiher~\cite{LMR}, determines $I(H, \beta)$ for the remaining nonmonochromatic 3-vertex graphs $H$; when $H$ is a nonmonochromatic triangle,   $I(H, \beta)$  is closely related to the problem of minimizing the number of triangles in a graph with a given density, which is answered by
a classical result of Razborov~\cite{Razborov}. 

Consequently,  the smallest graphs $H$ for which $I(H, \beta)$ is not known are graphs on four vertices. Perhaps the first interesting case is the four vertex $3$-edge alternating path $AP_4$, which consists of vertices $u,v,w,x$ where $uv$ and $wx$ are red and $vw$ is blue (see Figure~\ref{fig:AP_4}).  More generally, let $AP_k$ be the red/blue colored $(k+1)$-vertex path with $k$ edges where every two incident edges have different colors; if $k$ is even the end edges will have distinct colors, while if $k$ is odd, we further stipulate that the first edge and last edge are both red.

\tikzset{vertex/.style={inner sep=2pt, outer sep=0pt, circle, fill=black,draw}}

\begin{figure}[ht]
\begin{center}

\begin{tikzpicture}


  \node[vertex,label=below:$v$] (v) at (0,0) {};
  \node[vertex,label=below:$w$] (w) at (1.5,0) {};

  \node[vertex,label=above:$u$] (u) at (0,1.5) {};
  \node[vertex,label=above:$x$] (x) at (1.5,1.5) {};

  \draw[edge_color_blue] (v) -- (w);
  \draw[edge_color_red]        (u) -- (v);   
  \draw[edge_color_red]        (w) -- (x);   

 \node at (0.75,-1) {$AP_4$}; 


  \node[vertex,label=below:$v$] (a1) at (6,0) {};
  \node[vertex,label=below:$w$] (a2) at (7.5,0) {};
  \node[vertex,label=above:$x$] (a3) at (7.5,1.5) {};
  \node[vertex,label=above:$u$] (a4) at (6,1.5) {};

  \draw[edge_color_blue]        (a1) -- (a2);
  \draw[edge_color_red] (a2) -- (a3);
  \draw[edge_color_blue]        (a3) -- (a4);
  \draw[edge_color_red] (a4) -- (a1);
 \node at (6.75,-1) {$AC_4$}; 
\end{tikzpicture}
\end{center}
\caption{Alternating path $AP_4$ and Alternating 4-cycle $AC_4$.}
\label{fig:AP_4}
\end{figure}
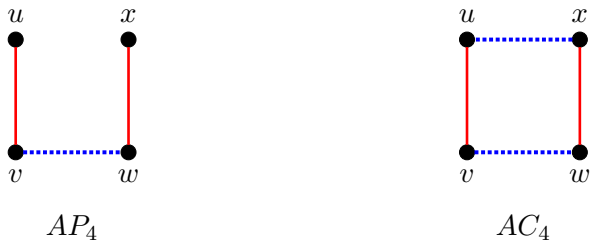
 Basit,  Granet,  Horsley,  K\"undgen and Staden~\cite[Corollary 1.4]{basit2025semiinducibilityproblem} determined 
  $\max(AP_k)$ when $k$ is even. Chen, Clemen and Noel~\cite[Theorem 1.2]{chen2025maximizingalternatingpathsentropy} determined $\max(AP_k)$ for all $k\ge 1$,
 and  Bodn\'ar and Pikhurko~\cite[Theorem 6.2]{bodnar2025exactinducibility} also determined $\max(AP_4)$. Despite these results about $\max(AP_k)$, the function $I(AP_k, \beta)$ was not known for any $k \ge 4$.

Our first result completely determines $I(AP_4, \beta)$, answering a question of 
 Basit,  Granet,  Horsley,  K\"undgen, and Staden~\cite[Problem 9.2]{basit2025semiinducibilityproblem}. The short proof we give in Section~\ref{sec:AP4} uses only double  counting of various quantities related to the vertex degrees in a graph.
\medskip

\begin{theorem} 
\label{thm:edge-nonedge-edge}
For all $\beta \in [0,1]$, we have $I(AP_4, \beta) = \beta^2(1-\beta)$.
Equality is achieved only for (asymptotically) regular graphs.
\end{theorem}

We generalize part of Theorem~\ref{thm:edge-nonedge-edge} to the family of double stars defined as follows. For an integer $s$, we define the alternating double-star $D(s)$ 
as the following labeled (and ordered) graph on $2s+2$ vertices: $V(D(s))=\{v_1,\ldots, v_s, v, u, u_1,\ldots, u_s\}$, and $E(D(s))=\{v_1v,\ldots, v_sv, uu_1,\ldots, uu_s\}$ and $uv\not\in E(D(s))$ (see Figure~\ref{fig:S5}).

\begin{figure}[ht]
\begin{center}
\begin{tikzpicture}
\draw
    (0,0) node[vertex,label=below:$v$](v){}
    (3.5,0) node[vertex,label=below:$u$](u){}
 ;   
    \foreach \i in {1,...,5}{
        \draw
        (v) ++(90+72*\i:1)node[vertex](v\i){} ++(90+72*\i:0.4) node{$v_\i$}
        (u) ++(90+72*\i:1)node[vertex](u\i){} ++(90+72*\i:0.4) node{$u_\i$}           
;
    \draw[edge_color_red]     
        (v) -- (v\i)
        (u) -- (u\i)    
        ;
    }
    
    \draw[edge_color_blue] (u)--(v);
\end{tikzpicture}
\end{center}
\caption{The graph $D(5)$.}
\label{fig:S5}
\end{figure}
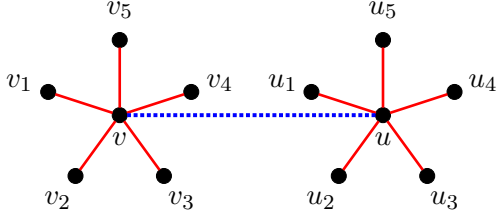

Clearly $D(1)=AP_4$.  The  following result determines $I(D(s), \beta)$
for a range of $\beta$ and also determines $\max(D(s))$.

\begin{theorem}\label{broom}
Fix $s \ge 1$ and $1-1/2s \le \beta\le 1$. Then $I(D(s), \beta) = \beta^{2s}(1-\beta)$. The equality holds for (asymptotically) regular graphs. Moreover,
$\max(D(s)) = (2s)^{2s}\cdot(2s+1)^{-2s-1}$.
\end{theorem}

Our next result concerns the alternating $4$-cycle, which is the graph obtained from $AP_4$ by adding a non-edge between its endpoints. Let us call this colored graph $AC_4$, which has vertices $u,v,w,x$, red edges $uv$ and $wx$ and blue edges $vw$ and $xu$ 
(see Figure~\ref{fig:AP_4}). The number of $AC_4$ in a graph $G$ is the same as the number of $AC_4$ in the complement of $G$. Therefore $I(AC_4, \beta)=I(AC_4, 1-\beta)$.

In~\cite{basit2025semiinducibilityproblem}, $\max(AC_4)$ was determined, and the question of determining $I(AC_4, \beta)$ for all $\beta$ was reiterated in Problem 9.3 of \cite{basit2025semiinducibilityproblem} (earlier the third author also had posed this question).  
Our result below answers this question whenever $\beta \in\{1/k, 1-1/k\}$ for every positive integer $k$.

\begin{theorem} \label{thm:altc4}
We have $I(AC_4, \beta) \le \beta^2(1-\beta)$, with equality if $\beta \in\bigcup_{k\in\mathbb{N}}\{1/k, 1-1/k\}$. 
\end{theorem}

Assume $\beta\le 1/2$. The proof of Theorem~\ref{thm:altc4} shows that equality holds only  if the underlying graph is almost regular and has only a few $3$-edge paths whose endpoints are not adjacent.
The only way to make such a construction is to take vertex-disjoint cliques of size $(\beta+o(1))n$. This forces $1/\beta$ to be an integer. We suspect that when $1/\beta$ is not an integer, roughly the same construction is optimal. The following question makes this precise, positing that the behavior of $I(AC_4, \beta)$ when $\beta$ is not of the form $1/k$ or $1-1/k$, is similar to that of the triangle density problem solved by Razborov~\cite{Razborov}.

\begin{question}\label{conj:altC4}
    For all $\beta<1/2$, and $AC_4$-good sequences $(G_n)$ with density $\beta$,  is it true that the supremum of  $\rho(AC_4, (G_n))$ is achieved if $G_n$ comprises $k=\lceil 1/\beta \rceil$ pairwise disjoint cliques, where $k-1$ of these cliques have (asymptotically) the same size, and the remaining clique has size at most $n/k$?
\end{question}

Liu, Mubayi, and Reiher~\cite[Problem 7.1.]{LMR}  asked if there is a red/blue colored {\em complete} graph $H$ for which  $I(H,\beta)$ has two global maxima (i.e., in the inducibility setting). 
We show the answer to their problem is yes if we do not require $H$ to be complete  (i.e., in the semi-inducibility setting).

Denote by $P_{EENN}$ a path on five vertices with pattern edge-edge-nonedge-nonedge (see Figure~\ref{fig:P_EENN}).
Define $K(a,n)$ to be the $n$-vertex graph consisting of a clique on $a$ vertices and $n-a$ isolated vertices.
Given graphs $G$ and $H$ on the same vertex set, recall that the \emph{edit distance} between $G$ and $H$ is the minimum number of edges that we need to delete or add to transform $G$ into $H$.
The following theorem implies that $I(P_{EENN},\beta)$ has maximum $27/252$ iff $\beta \in \{ 9/16, 7/16 \}$ and the asymptotic maximizer is $K(a,n)$ or its complement for $a = 3/4$.

\begin{figure}[ht]
\begin{center}
\begin{tikzpicture}
  \node[vertex,label=below:$ $] (v) at (0,0) {};
  \node[vertex,label=below:$ $] (w) at (1,0) {};
  \node[vertex,label=above:$ $] (u) at (2,0) {};
  \node[vertex,label=above:$ $] (x) at (3,0) {};
  \node[vertex,label=above:$ $] (y) at (4,0) {};
  \draw[edge_color_red]  (v) -- (w)--(u);   \draw[edge_color_blue] (u)--(x)--(y);
\end{tikzpicture}
\end{center}
\caption{Path $P_{EENN}$.}
\label{fig:P_EENN}
\end{figure}
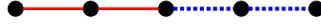

\begin{theorem}\label{thm:IPEEN}
 For all $\beta \in [0,1]$,    $I(P_{EENN},\beta) = \max\{ \beta^{3/2}-\beta^2, (1-\beta)^{3/2}-(1-\beta)^2\}$. 
 Moreover, for a fixed $\beta$ and $\varepsilon > 0$ there exist $\delta >0$ and  $n_0$ such for $n>n_0$,  every $n$-vertex red/blue clique $G$ with red density in $[\beta-\delta, \beta+\delta]$ maximizing $(\#P_{EENN},G)$ has edit distance at most $\varepsilon n^2$ from $K(a,n)$ or its complement for some $a$.
\end{theorem}

Denote by $P_{E^{\ell}N^k}$ a path of length $\ell+k$ with first $\ell$ edges red and the remaining $k$ edges blue. Notice $P_{EENN}$ is $P_{E^2N^2}$. 
Flag algebras experiments suggest that $K(a,n)$ and its complement are also extremal constructions for $P_{E^3N^3}$. We expect it is the case for longer symmetric paths as well.
\begin{conjecture}
    For $\ell \geq 2$, $I(P_{E^\ell N^\ell},\beta)$ is asymptotically achieved by $K(a,n)$ or its complement. 
\end{conjecture}

We next consider the case of monochromatic trees. Here we can dispense with mentioning colors and our problem is to simply maximize the number of copies of a given tree $T$ in a graph with density $\beta$. Even for the simple case when $T$ is a star, determining $I(T, \beta)$ is quite difficult. Indeed, depending on the value of $\beta$, the maximizers are $K(a,n)$ or the complement of $K(a,n)$ for the appropriate choice of $a$ that results in a graph of density $\beta$. This was proved by Reiher and Wagner~\cite{ReiherStar}.  We  generalize part of their result for arbitrary trees as long as $\beta$ is close to 1.

\begin{theorem} \label{thm:generaltrees} 
For every tree $T$ there exists $\beta_T \in (0,1)$ such that for $\beta>\beta_T$, the graph sequence $(G_n)$ where $G_n=K((1+o(1))\sqrt \beta \, n, n)$ achieves $I(T, \beta)$.
\end{theorem}

Note that we need $\beta>\beta_T$ for some $\beta_T$ in Theorem~\ref{thm:generaltrees}, since even for $T$ consisting of two edges and $\beta <1/2$, the complement of $K((1+o(1))\sqrt {1-\beta} n, n)$ has edge density $\beta$ and more copies of $T$ than $K((1+o(1))\sqrt \beta n, n)$.

If we consider $2$-edge-colored trees $T$, the problem appears to be much more complicated, perhaps even hopeless in general.
We pose the following question about $I(T, \beta)$ for certain values of $\beta$.

\begin{question}\label{ques:treegen}
    Let $T$ be a red/blue colored tree. Do there exist $\beta_T, \beta'_T$ with $0\le \beta_T<\beta'_T\le 1$ such that the following holds? Let $\beta \in [\beta_T, \beta'_T]$. Then $I(T, \beta)$  is asymptotically achieved by the following sequence of graphs $(G_n)$. The red subgraph graph $R$ of $G_n$ has the following form: there is a vertex partition $A \cup B=V(G_N)$, all vertices in $A$ have asymptotically the same  degree, all edges with one endpoint in $A$ and the other endpoint in $B$ are present, and $B$ is an independent set.
    \end{question}

It is possible that the minimum number of copies of $T$ in graphs with given density $\beta$ is also asymptotically achieved by the graphs in Question~\ref{ques:treegen}.
It was proved in~\cite{DHLMNPRS}
that 
when $T$ is monochromatic,  the minimum is achieved for regular graphs.

We conclude the introduction by stating a conjecture related to Question~\ref{ques:treegen} for the seemingly simple case of stars.  Let $\mathcal{S}$ be the class of graphs obtained from a regular graph by first adding universal vertices and then adding isolated vertices or by first adding isolated vertices and then adding universal vertices.
Notice that the initial regular graph may be a complete subgraph or a collection of isolated vertices.
Let $S_{a,b}$ be a star with $a$ edges (red) and $b$ non-edges (blue). 

\begin{question}\label{ques:stars}
Is it true that for each $a,b\geq 1$ and $\beta \in [0,1]$ there exists an $S_{a,b}$-good sequence $(G_n)$ with $G_n \in \mathcal{S}$ such that
\[
I(S_{a,b}, \beta) = \rho(S_{a,b},(G_n)) ?
\]
\end{question}
In Section~\ref{sec:stars} we answer Question~\ref{ques:stars} positively for $a=2, b=1$ and $\beta \ge 1/4$. More precisely, we prove that 

\begin{equation} \label{eqn:s21results}
I(S_{2,1}, \beta) = \begin{cases}
    \beta/4 & \quad\quad \text{if }\quad  1/4 \le \beta \le 1/2, \\
    \beta^2(1-\beta) &\quad\quad   \text{if } \quad 1/2 \le \beta \le\ 1.
\end{cases}
\end{equation}

As we have already mentioned, determining $I(H,\beta)$ includes determining $\max(H)$.
While determining $I(H,\beta)$ is wide open, $\max(H)$ has been determined for all 4-vertex graphs~\cite{bodnar2025} with the notable exception of a path with the pattern Edge-Edge-Nonedge. 
This seems to be a challenging open problem; see Problem 9.2 in~\cite{basit2025semiinducibilityproblem}
and Conjecture~1.1 in~\cite{bodnar2025}.

\begin{table}[]
    \centering
    \renewcommand{\arraystretch}{1.8}
    \begin{tabular}{|c|c|c|c|c|}
    \hline
       Graph $H$  &  $I(H,\beta)$   & Construction & Reference \\[1pt] \hline\hline
   $AP_4$  \Fe40{3 1 4 4 1 1}   &  $\beta^2(1 - \beta)$ \text{ if } $\beta \in [0,1]$  &   regular   & Theorem~\ref{thm:edge-nonedge-edge}    
\\[5pt] \hline
$D(s)$ \Fe60{1 1 1 1 4 4 1 1 1 4 1 3 1 1 4}
&
$\beta^{2s}(1 -\beta)$  if  $\beta \in [1-1/2s,1]$  & regular & Theorem~\ref{broom}
\\[5pt] \hline     
  $AC_4$   \Fe40{3 1 4 4 1 3}   &  $\beta^2(1 - \beta)$   \text{ if } $\beta \in \{1/k, 1-1/k\}$  &  regular   & Theorem~\ref{thm:altc4}    
\\[5pt] \hline
 $P_{EENN}$    \Fe50{1 1 1 4 3 1 1 3 1 4} & 
$      \begin{aligned}
            (1-\beta)^{3/2}-(1-\beta)^2  & \quad\text{if} \quad  \beta \in [0,1/2]\\
           \beta^{3/2}-\beta^2 & \quad\text{if} \quad  \beta \in [1/2,1]  
      \end{aligned}
      $
      &
\makecell{clique + isolated \\ clique $\cup$ isolated}
      &
      Theorem~\ref{thm:IPEEN}
\\[5pt]\hline
Tree $T$ & $\beta \in [\beta_T,1]$  & clique $\cup$ isolated &  Theorem~\ref{thm:generaltrees}
\\[5pt]\hline
     $S_{2,1}$     \Fe40{3 4 4 1 1 1} & 
     $
     \begin{aligned}
    \beta/4 &\quad \text{if}\quad  \beta\in[1/4,1/2] \\
    \beta^2(1-\beta) &\quad   \text{if} \quad \beta\in[1/2,1]    
\end{aligned}
$
&
\makecell{regular $\cup$ isolated \\ regular}
&
Section~\ref{sec:stars}
\\\hline
    \end{tabular}
    \caption{Summary of results from this paper. The depicted instance of $D(s)$ is $D(2)$.}
    \label{tab:summary}
\end{table}

Our results are summarized in Table~\ref{tab:summary}.
In Section~\ref{sec:AP4} we prove Theorem~\ref{thm:edge-nonedge-edge} describing $I(AP_4,\beta)$. 
Section~\ref{sec:doublestar} contains the proof of Theorem~\ref{broom}.
Theorem~\ref{thm:altc4}  and Question~\ref{conj:altC4} are discussed in Section~\ref{sec:AC4}.
Theorem~\ref{thm:IPEEN} is addressed in  Section~\ref{sec:PEENN} and 
Section~\ref{sec:generaltrees} contains Theorem~\ref{thm:generaltrees}.
Finally, Section~\ref{sec:stars} discusses possible extremal constructions $\mathcal{S}$ for stars.

\section{Alternating 3-edge paths}\label{sec:AP4}

We begin with a result of Nikiforov that is a very useful tool for counting $3$-edge paths in graphs. For completeness, we provide a proof.

\begin{lemma} [Nikiforov~\cite{Nikiforov2005}]\label{lem:nikiforov} 
Let $G=(V, E)$ be a graph on $n$ vertices with $m$ edges. Then
\[
\sum_{uv \in E} d_u d_v \;\ge\; \frac{4m^3}{n^2},
\]
where  $d_w$ denotes the degree of vertex $w$. Equality holds  iff $G$ is regular.
\end{lemma}

\begin{proof}
 We may assume $G$ has no isolated vertices. Set
\[
S:=\sum_{uv\in E} d_u d_v,\qquad
U:=\sum_{uv\in E}\frac{1}{\sqrt{d_u d_v}}.
\]
First note that for every edge $uv$,
\[
\frac{1}{\sqrt{d_u d_v}} \le \frac{1}{2}\!\left(\frac{1}{d_u}+\frac{1}{d_v}\right).
\]
Summing over all edges gives
\[
U \le \frac12\sum_{uv\in E}\Big(\frac{1}{d_u}+\frac{1}{d_v}\Big)
= \frac12\sum_{u} d_u\cdot\frac{1}{d_u} = \frac{n}{2}.
\]
For each of the $m$ edges $uv$, let $x_{uv} = (d_ud_v)^{1/4}$ and $y_{uv} = (d_ud_v)^{-1/4}$.
Let $x=(x_{uv})$ and $y=(y_{uv})$ be the corresponding vectors.   
Applying the Cauchy-Schwarz Inequality  $|x \cdot y| \le ||x||\cdot  ||y||$  to the vectors $x$ and $y$, we obtain
\begin{align*}
m^2\ &= \Big( \sum_{uv\in E} x_{uv}\cdot y_{uv}\Big)^2\ =\  |x \cdot y|^2\ \le\  
||x||^2 \cdot ||y||^2
=\Big(\sum_{uv\in E} x_{uv}^2\Big)\cdot 
\Big(\sum_{uv\in E} y_{uv}^2\Big)\\ &=\  \Big(\sum_{uv\in E}\sqrt{d_u d_v}\Big)\cdot \Big(\sum_{uv\in E}\frac{1}{\sqrt{d_u d_v}}\Big).
\end{align*}

Using $U\le n/2$, we obtain
\[
\sum_{uv\in E}\sqrt{d_u d_v} \ge \frac{m^2}{U} \ge \frac{2m^2}{n}.
\]
Combining this with an application of  the inequality between quadratic and arithmetic means  gives
\[
S=\sum_{uv\in E}d_u d_v \ge \frac{1}{m}\Big(\sum_{uv\in E}\sqrt{d_u d_v}\Big)^2 \ge \frac{1}{m}\left(\frac{2m^2}{n}\right)^{\!2}
= \frac{4m^3}{n^2},
\]
which completes the proof. Equality holds only if all the $d_u$ are the same due to our applications of  the arithmetic/quadratic mean inequality.
\end{proof}

Recall that $AP_4$ is  the alternating path of length three. A copy of $AP_4$ in $G$ is a (colored) subgraph of $G$ isomorphic to $AP_4$, and hence $(\#AP_4, G)$ is twice the number of copies of $AP_4$ in $G$. Using Lemma~\ref{lem:nikiforov}, we obtain an upper bound on the number of copies of $AP_4$ in a  graph with given edge density that is tight for regular graphs. This completely determines the profile for the number of $AP_4$ in a graph. Consequently,  Theorem~\ref{thm:edge-nonedge-edge}, which states  $I(AP_4, \beta) = \beta^2(1-\beta)$,  is an immediate consequence of Theorem~\ref{thm:D(1)} below. 

\begin{theorem} \label{thm:D(1)}
    Let $G=(V, E)$ be a graph with $m$ edges. Then
    $$\sum_{uv \not\in E}d_ud_v  \le 2m^2-\frac{2m^2}{n} - \frac{4m^3}{n^2}.$$
    Equality holds only if $G$ is regular. Moreover, the number of copies of $AP_4$ in $G$ is at most 
    $$2m^2-\frac{2m^2}{n} - \frac{4m^3}{n^2}-t,$$
     where $t$ is the number of $3$-element vertex sets in $G$ spanning exactly two edges.
\end{theorem}
\begin{proof}
    $$2m^2-\frac{1}{2}\sum_{v} d_v^2=\frac{1}{2}\left(\sum_{v} d_v\right)^2 - \frac{1}{2}\sum_{v} d_v^2=\sum_{uv \in \binom{V}{2}} d_ud_v= \sum_{uv \not\in E} d_ud_v +\sum_{uv \in E}d_ud_v.$$
Rearranging gives
$$\sum_{uv \not\in E} d_ud_v = 2m^2 - \left( \frac{1}{2}\sum_v d_v^2 +\sum_{uv \in E}d_ud_v\right).$$
Now, if $n$ and $m$ are fixed,  then $\sum_{uv \not\in E} d_ud_v$ is maximized when $(\sum_v d_v^2)/2 +\sum_{uv \in E}d_ud_v$ is minimized. The first quantity is minimized only for regular graphs by convexity, and the second quantity is also minimized only for regular graphs due to Lemma~\ref{lem:nikiforov}. Therefore $\sum_{uv \not\in E} d_ud_v$ is maximized only for regular graphs. The statement about $AP_4$ follows by observing that $\sum_{uv \not\in E}d_ud_v$ is the number of copies of $AP_4$ together with the number of $3$-element sets with exactly two edges (if the neighborhoods of $u$ and $v$ intersect).
\end{proof}

\subsection{Density of alternating paths of length $3$ via flag algebras}

Semi-inducibility problems are approachable by flag algebras as demonstrated by Chen and Noel~\cite{chen2025} and Bodn\'ar and Pikhurko~\cite{bodnar2025}.
Before finding the proof of Theorem~\ref{thm:edge-nonedge-edge} presented above, we had
  determined $I(H,\beta)$ for $1/2\le \beta \le 1$
   using flags on at most $4$ vertices.
Numerical experiments with several fixed values of $\beta < 1/2$ 
indicate that flag algebras may also be used  for $\beta < 1/2$. 
However, the proof would need flags on at least $6$ vertices instead of $4$.
We expect such a potential flag algebra proof parametrized by $\beta$ on 6 vertices to be significantly more complicated. 
After discovering the proof for the entire range without using flag algebras, we did not further attempt to find such a flag algebra proof.
We still include the easy flag algebra proof for $\beta\ge 1/2$ here for readers familiar with the method, but we skip the introduction to the standard notation to save space.

\begin{theorem}\label{thm:AP4flags}
$I(AP_4, \beta) = \beta^2(1-\beta)$ for $\beta \in [1/2,1]$.
\end{theorem}

\begin{proof}
Recall that in the definition of $I(AP_4,\beta)$ the subgraph count is scaled by $n^4$. However, the flag algebra densities of $4$-vertex subgraphs are scaled by $\binom{n}{4}$. Hence our aim is to prove an upper bound $24\beta^2(1-\beta)$. 
In flag algebras, the number of $AP_4$  can be written as
\[
O := 4 \Fuuuu222122 + 4 \Fuuuu222112 + 8 \Fuuuu221122 + 6 \Fuuuu221112 + 8 \Fuuuu211112.
\]
Our goal is to show $O \leq 24\beta^2(1-\beta)$.
Let $\alpha = 1-\beta$ be the density of non-edges.
We will combine $O$ with a linear combination of the following four squares and an expression fixing the density of non-edges. 
\begin{align*}
C_1 &:= \Bigg\llbracket  \Bigg(\Fllu221   -\Fllu212\Bigg)  ^{2} \Bigg\rrbracket\\
C_2 &:= \Bigg\llbracket  \Bigg(\Fllu121   -\Fllu112\Bigg)  ^{2} \Bigg\rrbracket\\
C_3 &:=  12\Bigg\llbracket  \Bigg(  -\alpha \Fllu222 +  (1-2\alpha) \Fllu221 + (1-\alpha) \Fllu211 \Bigg)  ^{2} \Bigg\rrbracket\\
C_4 &:=  12\Bigg\llbracket  \Bigg(  - \alpha \Fllu122 +  (1-2\alpha) \Fllu121 + (1-\alpha) \Fllu111 \Bigg)  ^{2} \Bigg\rrbracket\\
E &:= 6\left(\Fuu1 - \alpha\right)
\end{align*}

Notice that all five terms are non-negative. 
They can be expanded as linear combinations of $4$-vertex graphs. The coefficients of these linear combinations when expanded on $4$-vertex flags are in Table~\ref{tab:FA}.

\begin{table}[h!]
\begin{center}

\begin{tabular}{ccccccc}
 & O & $C_1$ & $C_2$ & $C_3$ & $C_4$ & $E$ \\
\Fuuuu222222 &
0 &
0 &
0 &
$12\alpha^2$ &
0 &
$0 -6\alpha$
\\[10pt]
\Fuuuu222221 &
4 &
0 &
0 &
$10\alpha^2-4\alpha$ &
$2\alpha^2$ &
$1 -6\alpha$
\\[10pt]
\Fuuuu222211 &
4 &
1/6 &
0 &
$10\alpha^2-8\alpha+1$ &
$4\alpha^2 - 2\alpha$ &
$2 -6\alpha$
\\[10pt]
\Fuuuu221211 &
0 &
0 &
1/2 &
$6\alpha^2 - 6\alpha$ &
$12\alpha^2-12\alpha+3$ &
$3 -6\alpha$
\\[10pt]
\Fuuuu222111 &
0 &
1/2 &
0 &
$12\alpha^2-12\alpha+3$ &
$6\alpha^2-6\alpha$ &
$3 -6\alpha$
\\[10pt]
\Fuuuu221122 &
8 &
$-2/3$ &
0 &
0 &
$4\alpha^2$ &
$2 -6\alpha$
\\[10pt]
\Fuuuu221121 &
6 &
$-1/6$ &
$-1/6$ &
$4\alpha^2-6\alpha+2$ &
$4\alpha^2-2\alpha$ &
$3 -6\alpha$
\\[10pt]
\Fuuuu221111 &
0 &
0 &
1/6 &
$4\alpha^2-6\alpha+2$ &
$10\alpha^2 - 12\alpha + 3$ &
$4 -6\alpha$
\\[10pt]
\Fuuuu121121 &
8 &
0 &
$-2/3$ &
$4\alpha^2 - 8\alpha + 4$ &
0 &
$4 -6\alpha$
\\[10pt]
\Fuuuu121111 &
0 &
0 &
0 &
$2\alpha^2-4\alpha+2$ &
$10\alpha^2 - 16\alpha + 6$ &
$5 -6\alpha$
\\[10pt]
\Fuuuu111111 &
0 &
0 &
0 &
0 &
$12\alpha^2-24\alpha+12$ &
$6 -6\alpha$
\end{tabular}
\end{center}
\caption{Coefficients of expansions into unlabeled graphs on 4 vertices.}\label{tab:FA}
\end{table}

The following linear combination yields that the coefficient for each graph on $4$ vertices is equal to $24\alpha(1-\alpha)^2$. 
\begin{align*}
O  
+ (48\alpha^3 - 96\alpha^2 + 48\alpha)C_1   
+ (48\alpha^3 - 72\alpha^2 + 24\alpha + 12)C_2 
+ (-4\alpha + 4)C_3 \\
+ (-4\alpha + 2)C_4
+ (-12\alpha^2 + 16\alpha - 4)E
= 24\alpha(1-\alpha)^2.
\end{align*}
Notice that the proof works  only  for $0 \leq \alpha \leq 1/2$
because the factors for $C_1, C_2, C_3$ and $C_4$ in the linear combination need to be non-negative,
however, the factor for $C_4$ is negative for $\alpha > 1/2$.
Since $\alpha = 1-\beta$, the result follows.
\end{proof}

\section{Density of alternating double stars}\label{sec:doublestar}
In this section we prove Theorem~\ref{broom}. This follows from the following theorem.
 Recall, that for an integer $s$,  the alternating double-star $D(s)$ 
is the following labeled (and ordered) graph on $2s+2$ vertices: $V(D(s))=\{v_1,\ldots, v_s, v, u, u_1,\ldots, u_s\}$, and $E(D(s))=\{v_1v,\ldots, v_sv, uu_1,\ldots, uu_s\}$ and $uv\not\in E(D(s))$ (see Figure~\ref{fig:S5}).

\begin{theorem} Fix $s \ge 1$. Then \begin{equation}\label{t31}\max(D(s)) \le \frac{(2s)^{2s}}{(2s+1)^{2s+1}}.
\end{equation}
If  $G$ is an $n$-vertex graph with average degree  $d\ge (1-1/(2s))n$, then $(\#D(s), G)\le (n-d)\cdot n\cdot  d^{2s}$. 
\end{theorem}
\begin{proof}
Let $G$ have  average degree $d$. 
We have
\begin{equation}
\label{degsum} 
(\#D(s), G) \le 2\sum_{uv\notin E} d_u^{s}d_v^{s}
\le  \sum_{uv\notin E}\left( d_u^{2s}+ d_v^{2s} \right)=\sum_{u} d^{2s}_u\cdot (n-1-d_u)<\sum_{u} d^{2s}_u\cdot (n-d_u). 
\end{equation}

It is easy to give an upper bound:

\begin{equation}\label{degglob} 
\sum_{u} d^{2s}_u\cdot (n-d_u)\le \frac{1}{2s}\sum_{u} d^{2s}_u\cdot (2sn-2sd_u)\le \frac{n}{2s}
\Big(\frac{2sn}{2s+1}\Big)^{2s+1} =\frac{(2s)^{2s}}{(2s+1)^{2s+1}}n^{2s+2},
\end{equation}
yielding $\max(D(s)) = (2s)^{2s}\cdot(2s+1)^{-2s-1}$.

To give an upper bound on the sum in~\eqref{degsum}, as a  function of the density of the graph, we use a well-known general optimization trick (and provide its proof).

\begin{lemma}\label{opt}
For an  $\alpha>0$ define  $X_{\alpha}$ to be  the set of vectors $(x_1, \ldots, x_n)$ such that $x_i \in \{0, \alpha\}$  for all $i$ except for possibly one coordinate.
    Suppose that $0<\gamma<1$ and $f: [0,1] \to \mathbb {R}$  is convex in $[0,\gamma]$ and concave in $[\gamma,1]$. Then  with the restriction that $x_i\in [0,1]$ and
    $\sum x_i = D$ for some $0\le D\le n$, $
    \sum_{i=1}^n f(x_i)$ is maximized at some element of $X_{\alpha}$, for some $\alpha\ge \gamma$.
\end{lemma}
\begin{proof} Suppose that $x=(x_1, \ldots, x_n)$ maximizes the sum subject to $\sum x_i=D$. Then if
  there were $0<x_i\le x_j<\gamma$, then we can decrease $x_i$ and increase $x_j$ with the same amount to increase the sum $\sum_{i=1}^n f(x_i)$ (by convexity), so there are no two such $i,j$.
Similarly, if 
$\gamma\le x_i<x_j$ for some $i,j$, then  we can increase $x_i$ and decrease $x_j$ to increase the sum (by concavity), hence there are no two such $i,j$, i.e., all $x_i\ge \gamma$ are equal to each other (and we define $\alpha\ge \gamma$ to be their value, setting $\alpha =\gamma$ if there is no $x_i\ge \gamma$), and all but at most one of the remaining variables are $0$.  We conclude that $x \in X_\alpha$ as claimed. 
\end{proof}
We apply Lemma~\ref{opt}
with $f(x)=(1-x)x^{2s}= x^{2s}-x^{2s+1}$. Using $$f'(x)=2s\cdot  x^{2s-1}-(2s+1)\cdot x^{2s}\quad \quad \text{ and }  \quad \quad f''(x)=2s(2s-1)\cdot  x^{2s-2}-2s(2s+1)\cdot x^{2s-1},$$
we have $f''(x)>0$ iff $\frac{2s-1}{2s+1}> x$, i.e., we can choose $\gamma=\frac{2s-1}{2s+1}$ and $D=d$. Note that we define $x_u=d_u/n$.
Using Lemma~\ref{opt}, there is an $\alpha> \gamma$ such that the maximum of $f$ is achieved when $m$ variables are $\alpha$, (for some $m$) at most one is between $0$ and $\gamma$, and the rest of the variables are $0$. Observe that $\alpha\ge d/n$ also holds. 
Ignoring that one outlier variable, whose role is negligible for large $n$, we have that $\alpha m = d=D$,
and we have the following  upper bound from Lemma~\ref{opt}:
\[
\sum_{i=1}^nf(x_i) \leq
m  \cdot f(\alpha) 
=
\frac{d}{\alpha}(
\alpha^{2s}-\alpha^{2s+1}
).
\]
This implies the following upper bound on \eqref{degsum}, which essentially proves~\eqref{t31}:
 
\[ n^{2s+1}\frac{d}{\alpha}\left(\alpha^{2s}-\alpha^{2s+1}\right)= n^{2s+1}\cdot d\cdot \alpha^{2s-1}\cdot (1-\alpha). \]
Now assume that  $d\ge (1-1/(2s))n$.
Then $\alpha\ge d/n \ge (2s-1)/(2s)$, and using that $\alpha^{2s-1}\cdot (1-\alpha)$ is decreasing when $\alpha> (2s-1)/2s$, we have   

\begin{equation} \label{eqn:prop3.3}
n^{2s+1}\cdot d\cdot \alpha^{2s-1}\cdot (1-\alpha)\le n\cdot (n-d)\cdot d^{2s} \le  n^{2s+2}\cdot  \frac{(2s)^{2s}}{(2s+1)^{2s+1}}. 
\end{equation}
\end{proof}

When $s=1$, Theorem~\ref{thm:edge-nonedge-edge} shows that the stronger bound $d^2(n-d)n$ holds for every $d$, while a similar statement does not hold for $s\ge 2$, as  a smaller regular  graph with additional isolated vertices contains more copies of $D(s)$ than a regular graph.

We shall later use the following proposition which is the $s=1$ case of~\eqref{degsum} and the first inequality in (\ref{eqn:prop3.3}).

\begin{proposition}\label{degree_sum}
Let $G$ be a graph with degree sequence $d_1,\ldots, d_n,$ and average degree $d\ge n/2$. Then 
$$\sum_i d_i^2(n-d_i)\le nd^2(n-d).$$
\end{proposition}

\section{Alternating $C_4$}\label{sec:AC4}
In this section we present a relatively short proof of Theorem~\ref{thm:altc4} and then discuss the case when the density $\beta \ne 1/k$. 

\bigskip

\noindent
{\bf Proof of Theorem~\ref{thm:altc4}.}
Recall that we are given an $n$-vertex graph $G$ with density $\beta+o(1)$ and we plan to prove that $(\#AC_4, G) \le (1+o(1))\beta^2(1-\beta)n^4$. Our main (simple) inequality is
$$(\#AC_4,G) \le 2\sum_{uv \not\in E} d_ud_v.$$
An $AC_4$ can be obtained by choosing a non-edge $uv$ and neighbors $u'$ of $u$ and $v'$ of $v$. 
If $u'\neq v'$ and $u'v' \not\in E$, then $uu'v'v$ yields a copy of $AC_4$.
Since each $AC_4$ has two non-edges, this counts every copy twice. Moreover, $(\#AC_4, G)$ is four times the number of copies of $AC_4$, giving the inequality.
More precisely, if we write $t$ for the number of $3$-element vertex sets that span exactly two edges and $s$ for the number of $3$-edge paths whose endpoints form a non-edge, then 
$$\sum_{uv \not\in E} d_ud_v = t+s+\frac{1}{2}\cdot (\#AC_4, G).$$
Let $m=|E(G)|=(1+o(1))\beta \binom{n}{2}$. Applying Theorem~\ref{thm:D(1)} 
$$(\#AC_4,G) \le  2\sum_{uv \not\in E} d_ud_v  \le 4m^2-\frac{4m^2}{n} - \frac{8m^3}{n^2} \le 4m^2 - \frac{8m^3}{n^2}
=\left(1+o(1)\right)\beta^2(1-\beta)n^4.$$
The first inequality above is asymptotically  sharp if $s=o(n^4)$, and the second inequality is asymptotically sharp if $G$ is close to regular. Hence, if $G$ is the disjoint union of cliques of size $(\beta+o(1))n$, then $(\#AC_4, G) = (1+o(1))\beta^2(1-\beta)n^4$ as $G$ is close to regular and $s=0$. Such a $G$ exists only if $k=1/\beta$ is an integer, completing the proof of the theorem.\qed

\bigskip

If $1/\beta$ is not an integer, then it is impossible to construct a regular graph with the properties discussed in the proof of Theorem~\ref{thm:altc4}.  Here we suggest that the extremal construction may be obtained by again taking disjoint cliques where all but one of them have the same size, and the remaining clique is smaller, see Question~\ref{conj:altC4}.

We tried to compare flag algebra calculations on $7$ vertices and this construction for many values of $\beta$, and the numbers were close but not equal (see Figure~\ref{fig:C4profile}). For example, even computing on $8$ vertices we obtained a numerical upper bound  $I(AC_4,\frac25)  \leq 0.0859499752$ (improving the bound obtained from a computation on $7$ vertices) while the construction gives $I(AC_4,\frac25) \geq 0.08566600788$.

The construction in this case (three disjoint cliques) is simple enough that one might expect that a flag algebra computation on a moderate number of vertices can ``see'' it if it is in fact uniquely extremal. Maybe there are other constructions we are not aware of, matching or beating our construction. On the other hand, such constructions can certainly not be regular,
since imposing regularity constraints to our flag algebra calculation gives a  numerical upper bound of $0.08409772$.


\begin{figure}
\begin{center}
\begin{tikzpicture}
\begin{axis}[ymin=0.0,ymax=0.1504,xmin=0, xmax=0.5,enlargelimits=false,    xlabel={$\beta$},
    ylabel={$I(AC_4,\beta)$},    
    legend pos=north west
    ]
    \addplot [blue,fill=blue!10!white,fill opacity=1,
    ] coordinates {
( 0.0100000000000000 , 0.0000990000041666668 )
( 0.0200000000000000 , 0.000392000012500000 )
( 0.0300000000000000 , 0.000873000000000000 )
( 0.0400000000000000 , 0.00153600000000000 )
( 0.0500000000000000 , 0.00237500000000000 )
( 0.0600000000000000 , 0.00338400000000000 )
( 0.0700000000000000 , 0.00455700000000000 )
( 0.0800000000000000 , 0.00588800000000000 )
( 0.0900000000000000 , 0.00737100000000000 )
( 0.100000000000000 , 0.00900000000000000 )
( 0.110000000000000 , 0.0107690000000000 )
( 0.120000000000000 , 0.0126720000000000 )
( 0.130000000000000 , 0.0147030000000000 )
( 0.140000000000000 , 0.0168560000000000 )
( 0.150000000000000 , 0.0191250000000000 )
( 0.160000000000000 , 0.0215040000000000 )
( 0.170000000000000 , 0.0239870000000000 )
( 0.180000000000000 , 0.0265680000000000 )
( 0.190000000000000 , 0.0292410000000000 )
( 0.200000000000000 , 0.0320000000000000 )
( 0.210000000000000 , 0.0348390000000000 )
( 0.220000000000000 , 0.0377520000000000 )
( 0.230000000000000 , 0.0407330000000000 )
( 0.240000000000000 , 0.0437760000000000 )
( 0.250000000000000 , 0.0468750000000000 )
( 0.260000000000000 , 0.0496050291666668 )
( 0.270000000000000 , 0.0523984208333332 )
( 0.280000000000000 , 0.0552684000000000 )
( 0.290000000000000 , 0.0582367875000000 )
( 0.300000000000000 , 0.0613967791666668 )
( 0.310000000000000 , 0.0648150833333332 )
( 0.320000000000000 , 0.0685419500000000 )
( 0.330000000000000 , 0.0726243583333332 )
( 0.340000000000000 , 0.0744246916666668 )
( 0.350000000000000 , 0.0755490458333332 )
( 0.360000000000000 , 0.0771899333333332 )
( 0.370000000000000 , 0.0792449333333332 )
( 0.380000000000000 , 0.0816247000000000 )
( 0.390000000000000 , 0.0842997000000000 )
( 0.400000000000000 , 0.0874120000000000 )
( 0.410000000000000 , 0.0901742125000000 )
( 0.420000000000000 , 0.0931719708333332 )
( 0.430000000000000 , 0.0964483458333332 )
( 0.440000000000000 , 0.0998207125000000 )
( 0.450000000000000 , 0.103408066666667 )
( 0.460000000000000 , 0.107221520833333 )
( 0.470000000000000 , 0.111273270833333 )
( 0.480000000000000 , 0.115576841666667 )
( 0.490000000000000 , 0.120147054166667 )
( 0.500000000000000 , 0.125000000000000 )
( 0.510000000000000 , 0.120147054166667 )
( 0.520000000000000 , 0.115576841666667 )
( 0.530000000000000 , 0.111273270833333 )
( 0.540000000000000 , 0.107221520833333 )
( 0.550000000000000 , 0.103408066666667 )
( 0.560000000000000 , 0.0998207125000000 )
( 0.570000000000000 , 0.0964483458333332 )
( 0.580000000000000 , 0.0931719708333332 )
( 0.590000000000000 , 0.0901742125000000 )
( 0.600000000000000 , 0.0874120000000000 )
( 0.610000000000000 , 0.0842997000000000 )
( 0.620000000000000 , 0.0816247000000000 )
( 0.630000000000000 , 0.0792449333333332 )
( 0.640000000000000 , 0.0771899333333332 )
( 0.650000000000000 , 0.0755490458333332 )
( 0.660000000000000 , 0.0744246916666668 )
( 0.670000000000000 , 0.0726243583333332 )
( 0.680000000000000 , 0.0685419458333332 )
( 0.690000000000000 , 0.0648150833333332 )
( 0.700000000000000 , 0.0613967791666668 )
( 0.710000000000000 , 0.0582367875000000 )
( 0.720000000000000 , 0.0552684000000000 )
( 0.730000000000000 , 0.0523984208333332 )
( 0.740000000000000 , 0.0496050291666668 )
( 0.750000000000000 , 0.0468750000000000 )
( 0.760000000000000 , 0.0437760000000000 )
( 0.770000000000000 , 0.0407330000000000 )
( 0.780000000000000 , 0.0377520000000000 )
( 0.790000000000000 , 0.0348390000000000 )
( 0.800000000000000 , 0.0320000000000000 )
( 0.810000000000000 , 0.0292410000000000 )
( 0.820000000000000 , 0.0265680000000000 )
( 0.830000000000000 , 0.0239870000000000 )
( 0.840000000000000 , 0.0215040000000000 )
( 0.850000000000000 , 0.0191250000000000 )
( 0.860000000000000 , 0.0168560000000000 )
( 0.870000000000000 , 0.0147030000000000 )
( 0.880000000000000 , 0.0126720000000000 )
( 0.890000000000000 , 0.0107690000000000 )
( 0.900000000000000 , 0.00900000000000000 )
( 0.910000000000000 , 0.00737100000000000 )
( 0.920000000000000 , 0.00588800000000000 )
( 0.930000000000000 , 0.00455700000000000 )
( 0.940000000000000 , 0.00338400000000000 )
( 0.950000000000000 , 0.00237500000000000 )
( 0.960000000000000 , 0.00153600000000000 )
( 0.970000000000000 , 0.000873000000000000 )
( 0.980000000000000 , 0.000392000020833333 )
( 0.990000000000000 , 0.0000990000125000000 )
( 1.00000000000000 , 0 )
};

    \addplot [only marks,
        mark=*,color=red,
    ] coordinates {
( 0.500000000000000 , 0.125000000000000 )
( 0.333333333333333 , 0.0740740740740740 )
( 0.250000000000000 , 0.0468750000000000 )
( 0.200000000000000 , 0.0320000000000000 )
( 0.166666666666667 , 0.0231481481481482 )
( 0.142857142857143 , 0.0174927113702624 )
( 0.125000000000000 , 0.0136718750000000 )
( 0.111111111111111 , 0.0109739368998628 )
( 0.100000000000000 , 0.00900000000000000 )
( 0.0909090909090909 , 0.00751314800901576 )
( 0.0833333333333333 , 0.00636574074074076 )
( 0.0769230769230769 , 0.00546199362767412 )
( 0.0714285714285714 , 0.00473760932944608 )
( 0.0666666666666667 , 0.00414814814814816 )
( 0.0625000000000000 , 0.00366210937500000 )
( 0.0588235294117647 , 0.00325666598819458 )
( 0.0555555555555556 , 0.00291495198902606 )
( 0.0526315789473684 , 0.00262428925499344 )
( 0.0500000000000000 , 0.00237500000000000 )
( 0.0476190476190476 , 0.00215959399632869 )
( 0.0454545454545455 , 0.00197220135236664 )
( 0.0434782608695652 , 0.00180816963918797 )
( 0.0416666666666667 , 0.00166377314814815 )
( 0.0400000000000000 , 0.00153600000000000 )
( 0.0384615384615385 , 0.00142239417387346 )
( 0.0370370370370370 , 0.00132093684905756 )
( 0.0357142857142857 , 0.00122995626822158 )
( 0.0344827586206897 , 0.00114805855098610 )
( 0.0333333333333333 , 0.00107407407407408 )
( 0.0322580645161290 , 0.00100701554160652 )
( 0.0312500000000000 , 0.000946044921875000 )
( 0.0303030303030303 , 0.000890447171438908 )
( 0.0294117647058824 , 0.000839609200081416 )
( 0.0285714285714286 , 0.000793002915451896 )
( 0.0277777777777778 , 0.000750171467764060 )
( 0.0270270270270270 , 0.000710718022624524 )
( 0.0263157894736842 , 0.000674296544685816 )
( 0.0256410256410256 , 0.000640604190900048 )
( 0.0250000000000000 , 0.000609375000000000 )
( 0.0243902439024390 , 0.000580374631824844 )
( 0.0238095238095238 , 0.000553395961559228 )
( 0.0232558139534884 , 0.000528255373740676 )
( 0.0227272727272727 , 0.000504789631855748 )
( 0.0222222222222222 , 0.000482853223593964 )
( 0.0217391304347826 , 0.000462316100928740 )
( 0.0212765957446809 , 0.000443061749323368 )
( 0.0208333333333333 , 0.000424985532407408 )
( 0.0204081632653061 , 0.000407993268111076 )
( 0.0200000000000000 , 0.000392000000000000 )
( 0.0196078431372549 , 0.000376928933818818 )
( 0.0192307692307692 , 0.000362710514337733 )
( 0.0188679245283019 , 0.000349281621741438 )
( 0.0185185185185185 , 0.000336584870192552 )
( 0.0181818181818182 , 0.000324567993989482 )
( 0.0178571428571429 , 0.000313183309037901 )
( 0.0175438596491228 , 0.000302387239258503 )
( 0.0172413793103448 , 0.000292139899134856 )
( 0.0169491525423729 , 0.000282404724923191 )
( 0.0166666666666667 , 0.000273148148148148 )
( 0.0163934426229508 , 0.000264339305933096 )
( 0.0161290322580645 , 0.000255949783491658 )
( 0.0158730158730159 , 0.000247953384763664 )
( 0.0156250000000000 , 0.000240325927734375 )
( 0.0153846153846154 , 0.000233045061447428 )
( 0.0151515151515152 , 0.000226090102123160 )
( 0.0149253731343284 , 0.000219441886136260 )
( 0.0147058823529412 , 0.000213082637899450 )
( 0.0144927536231884 , 0.000206995850950811 )
( 0.0142857142857143 , 0.000201166180758018 )
( 0.0140845070422535 , 0.000195579347938454 )
( 0.0138888888888889 , 0.000190222050754458 )
( 0.0136986301369863 , 0.000185081885881594 )
( 0.0135135135135135 , 0.000180147276568022 )
( 0.0133333333333333 , 0.000175407407407408 )
( 0.0131578947368421 , 0.000170852165038635 )
( 0.0129870129870130 , 0.000166472084164781 )
( 0.0128205128205128 , 0.000162258298352973 )
( 0.0126582278481013 , 0.000158202495137302 )
( 0.0125000000000000 , 0.000154296875000000 )
( 0.0123456790123457 , 0.000150534113852714 )
( 0.0121951219512195 , 0.000146907328680663 )
( 0.0120481927710843 , 0.000143410046048616 )
( 0.0119047619047619 , 0.000140036173199438 )
( 0.0117647058823529 , 0.000136779971504173 )
( 0.0116279069767442 , 0.000133636032047493 )
( 0.0114942528735632 , 0.000130599253154504 )
( 0.0113636363636364 , 0.000127664819684448 )
( 0.0112359550561798 , 0.000124828183934329 )
( 0.0111111111111111 , 0.000122085048010974 )
( 0.0109890109890110 , 0.000119431347543894 )
( 0.0108695652173913 , 0.000116863236623654 )
( 0.0107526881720430 , 0.000114377073861482 )
( 0.0106382978723404 , 0.000111969409475742 )
( 0.0105263157894737 , 0.000109636973319726 )
( 0.0104166666666667 , 0.000107376663773148 )
( 0.0103092783505155 , 0.000105185537426877 )
( 0.0102040816326531 , 0.000103060799496808 )
( 0.0101010101010101 , 0.000100999794908580 )
( 0 , 0 )
};

    \addplot [red,fill=red!10!white,fill opacity=1,
    ] coordinates {
( 0.500000000000000 , 0.000000000000000 )
( 0.500000000000000 , 0.125000000000000 )
( 0.459357277882798 , 0.106367544427014 )
( 0.427083333333333 , 0.0940936053240740 )
( 0.401600000000000 , 0.0861132800000000 )
( 0.381656804733728 , 0.0810239487412904 )
( 0.366255144032922 , 0.0778675337431624 )
( 0.354591836734694 , 0.0759872188671388 )
( 0.346016646848989 , 0.0749320284300016 )
( 0.340000000000000 , 0.0743925925925924 )
( 0.336108220603538 , 0.0741574907338328 )
( 0.333333333333333 , 0.0740740740740740 )
( 0.314878892733564 , 0.0662797380299564 )
( 0.299591836734694 , 0.0604748021657644 )
( 0.287037037037037 , 0.0561914151806128 )
( 0.276844411979547 , 0.0530701471218320 )
( 0.268698060941828 , 0.0508340942749060 )
( 0.262327416173570 , 0.0492694648361468 )
( 0.257500000000000 , 0.0482109375000000 )
( 0.254015466983938 , 0.0475305590246308 )
( 0.251700680272109 , 0.0471292825520232 )
( 0.250000000000000 , 0.0468750000000000 )
( 0.239506172839506 , 0.0430812376162172 )
( 0.230623818525520 , 0.0401040233561200 )
( 0.223177908555908 , 0.0377901752184210 )
( 0.217013888888889 , 0.0360145097897376 )
( 0.211995002082466 , 0.0346745707267259 )
( 0.208000000000000 , 0.0336864000000000 )
( 0.204921184159938 , 0.0329811338938784 )
( 0.202662721893491 , 0.0325022539476909 )
( 0.201139195443218 , 0.0322033599726050 )
( 0.200000000000000 , 0.0320000000000000 )
( 0.193239795918367 , 0.0298978208949396 )
( 0.187442289935365 , 0.0281981784579217 )
( 0.182520808561237 , 0.0268378126939646 )
( 0.178397012352772 , 0.0257630353671893 )
( 0.175000000000000 , 0.0249282407407408 )
( 0.172265520021500 , 0.0242946600354576 )
( 0.170135275754422 , 0.0238293173625722 )
( 0.168556311413454 , 0.0235041526478736 )
( 0.167480468750000 , 0.0232952833175659 )
( 0.166666666666667 , 0.0231481481481482 )
( 0.161951436845623 , 0.0218688578168927 )
( 0.157871972318339 , 0.0208142824858419 )
( 0.154379332073094 , 0.0199539235706202 )
( 0.151428571428571 , 0.0192612244897959 )
( 0.148978377306090 , 0.0187130477515439 )
( 0.146990740740741 , 0.0182892250371514 )
( 0.145430662413211 , 0.0179721694734544 )
( 0.144265887509131 , 0.0177465409855397 )
( 0.143466666666667 , 0.0175989570370370 )
( 0.142857142857143 , 0.0174927113702624 )
( 0.139381985535832 , 0.0166585186309052 )
( 0.136356353148534 , 0.0159614045475284 )
( 0.133750000000000 , 0.0153849609375000 )
( 0.131534827008078 , 0.0149146319414201 )
( 0.129684711481261 , 0.0145374998097858 )
( 0.128175352010451 , 0.0142420967652859 )
( 0.126984126984127 , 0.0140182395709606 )
( 0.126089965397924 , 0.0138568838974629 )
( 0.125473228772309 , 0.0137499959781222 )
( 0.125000000000000 , 0.0136718750000000 )
( 0.122333038757733 , 0.0130985439299180 )
( 0.120000000000000 , 0.0126145404663923 )
( 0.117980920178722 , 0.0122102876587856 )
( 0.116257088846881 , 0.0118771686421932 )
( 0.114810960804717 , 0.0116074282054265 )
( 0.113626075147125 , 0.0113940849821945 )
( 0.112686980609418 , 0.0112308530474751 )
( 0.111979166666667 , 0.0111120718496817 )
( 0.111488999893719 , 0.0110326435385824 )
( 0.111111111111111 , 0.0109739368998628 )
( 0.109000000000000 , 0.0105633000000000 )
( 0.107146358200176 , 0.0102139141285838 )
( 0.105536332179931 , 0.00991982721577672 )
( 0.104156847959280 , 0.00967562395180184 )
( 0.102995562130178 , 0.00947637635783936 )
( 0.102040816326531 , 0.00931759914850292 )
( 0.101281594873621 , 0.00919520939217772 )
( 0.100707485369901 , 0.00910549002968052 )
( 0.100308641975309 , 0.00904505685955732 )
( 0.100000000000000 , 0.00900000000000000 )
( 0.0982874766658551 , 0.00869597346278556 )
( 0.0967793367346939 , 0.00843567631732220 )
( 0.0954655807032657 , 0.00821522036164696 )
( 0.0943367189904586 , 0.00803103609027700 )
( 0.0933837429111531 , 0.00787984605543864 )
( 0.0925980975029727 , 0.00775864057213468 )
( 0.0919716560742202 , 0.00766465554775840 )
( 0.0914966963516231 , 0.00759535223882492 )
( 0.0911658781159523 , 0.00754839875693648 )
( 0.0909090909090909 , 0.00751314800901576 )
( 0.0894920720236496 , 0.00728184369587952 )
( 0.0882411263137022 , 0.00708279868548552 )
( 0.0871488033298647 , 0.00691336810424452 )
( 0.0862080000000000 , 0.00677110579200000 )
( 0.0854119425547997 , 0.00665374909517012 )
( 0.0847541695083390 , 0.00655920488479332 )
( 0.0842285156250000 , 0.00648553669452668 )
( 0.0838290968090860 , 0.00643095288323248 )
( 0.0835502958579882 , 0.00639379573544344 )
( 0.0833333333333333 , 0.00636574074074076 )
( 0.0821414438351518 , 0.00618571920772276 )
( 0.0810871018044108 , 0.00603014839720332 )
( 0.0801646090534979 , 0.00589716788881552 )
( 0.0793685121107266 , 0.00578504629449480 )
( 0.0786935904949651 , 0.00569217215279500 )
( 0.0781348456206679 , 0.00561704549951448 )
( 0.0776874902955334 , 0.00555827006118636 )
( 0.0773469387755102 , 0.00551454602249064 )
( 0.0771087973441980 , 0.00548466332274740 )
( 0.0769230769230769 , 0.00546199362767412 )
( 0.0759066358024691 , 0.00531916130681060 )
( 0.0750059453032105 , 0.00519528730448012 )
( 0.0742165509476450 , 0.00508902558436628 )
( 0.0735341755749919 , 0.00499911680910528 )
( 0.0729547114682250 , 0.00492438266376260 )
( 0.0724742128732940 , 0.00486372057045480 )
( 0.0720888888888889 , 0.00481609876543208 )
( 0.0717950967062848 , 0.00478055171218332 )
( 0.0715893351800554 , 0.00475617582618112 )
( 0.0714285714285714 , 0.00473760932944608 )
( 0.0705515088449532 , 0.00462239624220780 )
( 0.0697731755424063 , 0.00452216997861628 )
( 0.0690900239360623 , 0.00443593297986820 )
( 0.0684986380387758 , 0.00436274766668648 )
( 0.0679957280170879 , 0.00430173277658748 )
( 0.0675781250000000 , 0.00425205993652344 )
( 0.0672427761274642 , 0.00421295045479960 )
( 0.0669867398262460 , 0.00418367231734096 )
( 0.0668071813015168 , 0.00416353737446796 )
( 0.000000000000000 , 0.000000000000000 )
};

\addlegendentry{Flag Algebras}
\addlegendentry{Theorem~\ref{thm:altc4}}
\addlegendentry{Construction}
\end{axis}
\end{tikzpicture}\\
\end{center}
\caption{The upper was bound obtained by flag algebras calculation on 7 vertices and the lower bound obtained by a construction.}\label{fig:C4profile}
\end{figure}
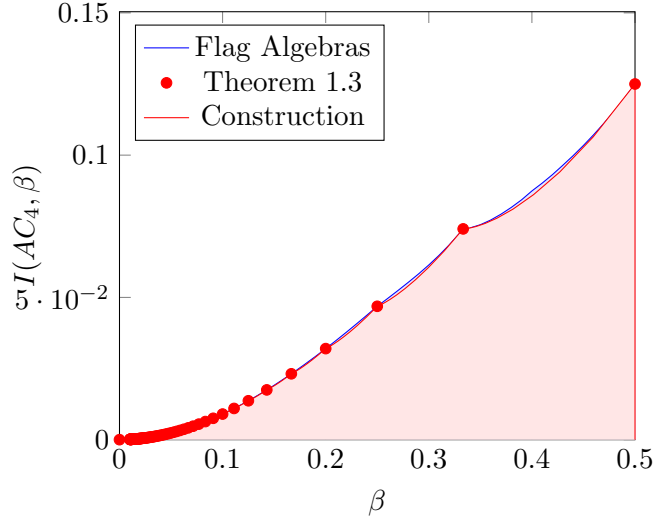

\section{Path $P_{EENN}$.}\label{sec:PEENN}

\begin{figure}[ht]
\begin{center}
\includegraphics[width=7cm]{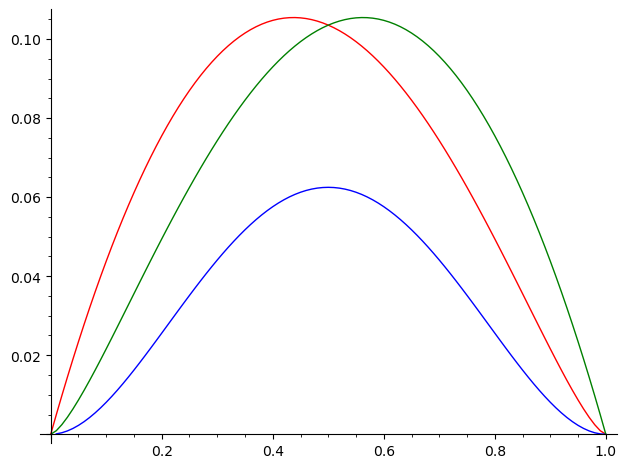}
\end{center}
    \caption{Red and green are bounds from clique and isolated vertices and the complement. Their maximum is $I(P_{EENN},\beta)$. 
    Blue is a regular graph, perhaps it is the minimum.}
    \label{fig:PEENN}
\end{figure}

Theorem~\ref{thm:IPEEN} is a corollary of the following two theorems that deals with  $\beta \in [1/2,1]$. 
See Figure~\ref{fig:PEENN} for an illustration of the resulting bound.
The case $\beta \in [0,1/2]$ follows from symmetry. 
Recall that $K(a,n)$ is the $n$-vertex graph comprising a clique on $a$ vertices and $n-a$ vertices of degree 0.

\begin{theorem}\label{thm:PEENNI}
    $I(P_{EENN},\beta) = \beta^{3/2}-\beta^2$ for $\beta\in [1/2,1]$. 
\end{theorem}
\begin{proof}
To simplify the expressions in the proof we let $a = \sqrt{\beta}$. 
Our goal is to show $I(P_{EENN},a^2) = a^3-a^4$ for $a \in [1/\sqrt{2},1]$.
The lower bound is implied by $K(a,n)$.
For a convergent sequence $(G_n)$ with the corresponding homomorphism $\phi$, the scaling for flag algebras  is
\[
\rho(P_{EENN}, G_n) = \frac{1}{120}\phi(P_{EENN}).
\]
In our flag algebra calculations, we omit writing $\phi$ since they are valid for all $\phi$.
We can rewrite $P_{EENN}$ as a linear combination of $5$-vertex subgraph densities as follows.

$
P_{EENN} = 
\Fe50{1 1 1 4 3 1 1 3 1 4}
=
 4  \cdot \Fe50{1 1 1 2 1 1 1 1 1 2}\,
+12 \cdot \Fe50{1 1 1 2 1 1 1 1 2 2}\,
+24 \cdot \Fe50{1 1 1 2 1 1 2 1 2 2}\,
+6  \cdot \Fe50{1 1 1 2 1 2 1 1 1 2}\,
+8  \cdot \Fe50{1 1 1 2 1 2 1 1 2 2}\,
+16 \cdot \Fe50{1 1 1 2 1 2 2 1 1 2}\,
+20 \cdot \Fe50{1 1 1 2 1 2 2 1 2 2}\,
+12 \cdot \Fe50{1 1 2 2 1 1 1 1 1 2}\,
+12 \cdot \Fe50{1 1 2 2 1 2 1 1 2 2}\,
+20 \cdot \Fe50{1 1 2 2 1 2 2 1 1 2}\,
+16 \cdot \Fe50{1 1 2 2 1 2 2 1 2 2}\,
+2  \cdot \Fe50{1 2 1 2 1 2 1 1 1 2}\,
+4  \cdot \Fe50{1 2 1 2 1 2 1 1 2 2}\,
+8  \cdot \Fe50{1 2 1 2 1 2 2 1 2 2}\,
+8  \cdot \Fe50{1 2 2 2 1 2 1 1 2 2}\,
+6  \cdot \Fe50{1 2 2 2 1 2 2 1 2 2}\,
+8  \cdot \Fe50{2 1 1 2 1 2 1 1 1 2}\,
+4  \cdot \Fe50{2 1 1 2 1 2 1 1 2 2}\,
+24 \cdot \Fe50{2 1 2 2 1 2 2 1 1 2}\,
+12 \cdot \Fe50{2 1 2 2 1 2 2 1 2 2}\,
+2  \cdot \Fe50{2 2 1 2 1 2 1 1 2 2}\,
+12 \cdot \Fe50{2 2 2 2 1 1 2 1 2 2}\,
+4  \cdot \Fe50{2 2 2 2 1 2 2 1 2 2}\,
$

In the flag algebra language, our main goal is to prove that
\begin{align}\label{eq:Peenn}
\frac{1}{120}\Fe50{1 1 1 4 3 1 1 3 1 4}  \leq a^3-a^4.
\end{align}
Since $\beta = a^2$ is the density of edges,  we can use $\Fe202 - a^2 = 0$.
The inequality \eqref{eq:Peenn} follows from
\begin{align}    
  0 &\geq a^2(1-a^2)\cdot  \Bigg(    \Fe50{1 1 1 4 3 1 1 3 1 4}  - 120(a^3-a^4) \Bigg) 
  +
  \Bigg( 120a^2(a^3-a^4) \cdot \Fe30{2 2 2} - 120(1-a^2)(a^3-a^4) \cdot \Fe30{1 1 1}  \nonumber  \\
  &~ +(-12a^6+120a^5+80a^4-80a^3+20a^2-20a) \cdot \Fe30{1 1 2} + 15\cdot B \cdot \Fe30{1 2 2} \Bigg)  \times  \Bigg(  \Fe20{2} - a^2  \Bigg) \label{eq:FAPEENN} \\
  &~ 
  + 60(a-a^2)\cdot  \Bigg\llbracket  \Bigg(  a \Fe43{1 1 1 2 1 1} - (1-a) \Fe43{1 1 1 2 2 2} \Bigg)  ^{2} \Bigg\rrbracket + 30\cdot C \cdot \Bigg\llbracket  \Bigg(  a \Fe43{1 2 2 2 2 2} - (1-a) \Fe43{1 2 1 2 1 2} \Bigg)  ^{2} \Bigg\rrbracket  ,
    \nonumber 
\end{align}
where 
$B = C = \sqrt{2}-1$ if $\beta \in [1/2,0.8^2]$
and $B=0.361, C=0$ if $\beta \in (0.8^2,1]$.
We remark that there is some freedom in choosing $B$ and $C$ when $\beta > 1/2$, but they are fixed for $\beta=1/2$.
A crucial point for the validity of the inequality~\eqref{eq:FAPEENN} is that both $(a-a^2)$ and $C$ i.e., the multiplicative factors for sum of squares, are non-negative. 

The evaluation of \eqref{eq:FAPEENN} is similar to the proof of Theorem~\ref{thm:AP4flags}. However, it is performed on 5-vertex graphs. Let $\mathcal{F}_5$ be the set of all $5$-vertex $2$-edge colored graphs up to isomorphism. A tedious yet straightforward enumeration reveals that $|\mathcal{F}_5|=34$.
Also, the corresponding coefficients are more involved.
On the other hand, they need to be tight only for subgraphs of the extremal constructions.
More precisely, the right hand side of \eqref{eq:FAPEENN} can be expressed as a linear combination of the elements of $\mathcal{F}_5$
\begin{align}\label{eq:PEENN0}
\sum_{F \in \mathcal{F}_5} c_{a,F} \cdot F \leq \max_{F \in \mathcal{F}_5} c_{a,F} = 0.
\end{align}
The interested reader may find all the coefficients $ c_{a,F}$ in the in the Appendix of the arXiv preprint of this paper.
\end{proof}

The proof of Theorem~\ref{thm:PEENNI} also implies the (asymptotic) stability of the extremal constructions.

\begin{theorem} 
Let $\beta \in [1/2,1)$ be fixed.
If $G$ is a sufficiently large graph with edge density $\beta$ maximizing the number of copies of $P_{EENN}$, then $G$ is in edit distance at most $\varepsilon n^2$ from a graph $K(\sqrt{\beta}n,n)$ or its complement if $\beta = 1/2$. 
\end{theorem}

\begin{proof}
If $(G_n)$ is a convergent sequence of graphs with edge densities $\beta$ achieving $I(P_{EENN},\beta)$,
all the 5-vertex subgraphs with non-zero density must have $c_{a,F} = 0$ in \eqref{eq:PEENN0}.
The inspection of  $c_{a,F}$ in \eqref{eq:PEENN0}
shows that $c_{a,F} = 0$ for all $a \in [\sqrt{\beta},1)$ only for $F$ in 
\[
\mathcal{A} := \left\{
\Fe50{1 1 1 1 1 1 1 1 1 1}, \quad
\Fe50{1 1 1 1 1 1 1 1 1 2}, \quad 
\Fe50{1 1 1 1 1 1 1 2 2 2}, \quad
\Fe50{1 1 1 1 2 2 2 2 2 2}, \quad
\Fe50{2 2 2 2 2 2 2 2 2 2}
\right\}.
\]
Additional graphs $F$ with at least one $a \in [\sqrt{\beta},1)$ such $c_{a,F} = 0$ appear only at $a=1/\sqrt{2}$ and they are the following four
\[
\mathcal{A}_{1/2} := \left\{
\Fe50{2 2 2 2 2 2 2 2 2 1}, \quad 
\Fe50{2 2 2 2 2 2 2 1 1 1}, \quad
\Fe50{2 2 2 2 1 1 1 1 1 1}, \quad
\Fe50{2 1 1 2 2 1 1 2 1 2} 
\right\}.
\]
However, $c_{1/\sqrt{2},C_5} = 0$ is an artifact of the proof.
A separate flag algebra calculation taylored for $a = 1/\sqrt{2}$ has  $c_{a,C_5} < 0$.
We conclude that if $(G_n)$ is a convergent sequence achieving $I(P_{EENN},\beta)$ for $\beta \in [1/2,1)$, then the densities of the following graphs
\[
 \mathcal{F}=\left\{
\begin{tabular}{ccccc}
 \Fuuuu111122, 
 &
 \Fuuuu112211,
 &
 \Fuuuu112212,
 &
 \Fuuuu112222,
 &
 \Fuuuu122221
 \\
 (a) & (b) &  (c) &(d)& (e) 
\end{tabular}
\right\}
\]
tend to zero since they do not appear as induced subrgaphs of $\mathcal{A} \cup \mathcal{A}_{1/2}\setminus\{C_5\}$.

Let $G$ be a sufficiently large extremal example.
After applying the induced removal lemma to $G$, we get  a graph  $G'$ in edit distance $\varepsilon n^2$ from $G$, which contains none of the graphs in $\mathcal{F}$ as induced subgraphs.

Let $K$ be the vertex set of a maximum clique in $G'$, let $I$ be the set of isolated vertices in $G'$, and let $R=V(G')-K-I$ be the remaining vertices.

\begin{claim}
    There is at most one vertex $k$ in $K$ such that all the vertices $r$ in $R$ are adjacent to all the vertices in $K\setminus\{k\}$ and not adjacent to $k$.
\end{claim}
\begin{proof}
    By the maximality of $K$, each vertex in $R$ must have a non-neighbor in $K$.
    Let $r$ be a vertex in $R$. Suppose for contradiction that $r$ is not adjacent to two vertices $u,v \in K$. Since $r$ is not isolated, it has a neighbor $w$.
    The four vertices $u,v,w,r$ induce one of (b),(c) or (d), giving a contradiction. Hence for each vertex in $r$, there is at most one vertex $u$ in $K$ such that $ur$ is not an edge.
    If $|R|=1$, then  the claim is proved.

    Let $r_1$ and $r_2$ be two distinct vertices in $R$. Suppose for contradiction that $k_1$ and $k_2$ are two distinct vertices in $K$ such that $r_1k_1$ and $r_2k_2$ are not edges. 
    If $r_1r_2$ is an edge, then $r_1r_2k_1k_2$ induce (e) and if $r_1r_2$ is not an edge $r_1r_2k_1k_2$ induce (c). 
    In both cases we obtain a contradiction.
\end{proof}

\begin{claim}
   The  graph induced by $R$ is an independent set.
\end{claim}
\begin{proof}
    If $r_1r_2$ is an edge in $G[R]$, then $K\setminus{k} \cup \{r_1,r_2\}$ is a clique contradicting the maximality of $K$.
\end{proof}

\begin{claim}
    Either $R=\emptyset$ or $I=\emptyset$.
\end{claim}
\begin{proof}
    Suppose for contradiction that $r\in R$ and $i \in I$ exist.
    Let $k$ be a non-neighbor of $r$ in $R$ and $u$ be a neighbor of $r$ in $K$.
    Since $\{r,i,k,u\}$ induces (a), we obtain a contradiction.
\end{proof}
The claims imply the statement of the theorem.
\end{proof}

\section{General trees}\label{sec:generaltrees}

In this section we prove Theorem~\ref{thm:generaltrees}. 
We begin by stating the following result of Reiher and Wagner~\cite{ReiherStar}.

\begin{theorem}  [Reiher-Wagner~\cite{ReiherStar}]\label{RWthm}
    Let $S_k$ be the star with $k$ edges, $\beta \in [0,1]$ and $\eta = 1-\sqrt{1-\beta}$. Then
    $$I(S_k, \beta) \le  \max \{ \beta^{(k+1)/2}, \eta+(1-\eta)\eta^k\}.$$ 
\end{theorem}

We note that the two quantities in the maximum are obtained when $G_n$ is $K(a,n)$ or its complement, for an appropriate $a$; moreover, for $\beta$ sufficiently close to 1, the first term in the maximum, which corresponds to $\rho(S_k, (K(a,n)))$,  is larger.
\bigskip

{\bf Proof of Theorem ~\ref{thm:generaltrees}.}
Recall that we are given a tree $T$ and our aim is to show that there exists $\beta_T \in (0,1)$ such that for $\beta>\beta_T$, the graph sequence $(G_n)$ where $G_n=K((1+o(1))\sqrt \beta \, n, n)$ achieves $I(T, \beta)$.
   We proceed by  induction on the number of edges of $T$. For the base case, we use Theorem~\ref{RWthm} which applies for all stars. For the induction step, let us assume $T$ is not a star. Then there is an edge $e$ of $T$ such that $T-e$ consists of two trees $T_1$ and $T_2$, each with at least one edge. By induction there exist $\beta_{T_1}$ and $\beta_{T_2}$ so let $\beta_T= \max\{\beta_{T_1}, \beta_{T_2}\}$. Consider $G$ with density $\beta>\beta_T$.
     Then clearly $(\#T, G) \le (\#T_1, G) \cdot (\#T_2, G)$. Both of the latter quantities are asymptotically maximized when $G$ is a  clique and isolated vertices by induction. Moreover, for such a choice of $G$ we have $(\#T, G) \ge  (\#T_1, G)\cdot  (\#T_2, G)$  up to lower order terms. Indeed, for every two disjoint copies of $T_1$ and $T_2$ in $G$, and for every two vertices $v_1$ and $v_2$ in $T_1$ and $T_2$, respectively, the pair $v_1v_2$ is adjacent in $G$. An appropriate choice of $v_1$ and $v_2$ yields a copy of $T$. 
\qed

\section{Stars}\label{sec:stars}

In this section we investigate Question~\ref{ques:stars}.
Here, we have edges as red and non-edges as blue. Let $S_{a,b}$ be a star with $a$ edges (red) and $b$ non-edges (blue), where $a \geq b$ and $a \geq 2$. 
We will consider 
\[
S_{2,1} = \Fuuuu443111 
\]
as the main example.

Recall from the introduction that $\mathcal{S}$ is the class of graphs obtained from a regular graph by first adding universal vertices and then adding isolated vertices or by first adding isolated vertices and then adding universal vertices.
The class $\mathcal{S}$ is a combination of several different constructions that seem to appear to be maximizing (or minimizing) $\#S_{a,b}$ for $n$-vertex graphs for some $\beta$.
Here we list these special cases of $\mathcal{S}$ and their  resulting densities. 

    (i)  The disjoint union of $(1-s)n$ isolated vertices and a regular graph $R$ on $sn$ vertices for some $s\in (0,1)$. 
    The graph $R$ has density  $x$,
    resulting in    degree $snx$ for every vertex in $R$.
    This implies $\beta = s^2x$, solving for $x$ yields $x = \beta/s^2$.
    The density of labeled $S_{a,b}$ as a function of $s$ is then
    \[
    f(s) = s(sx)^a(1-sx)^b = s\left( \frac{\beta}{s} \right)^a\left(1-\frac{\beta}{s}\right)^b.
    \]
    By solving  $f'(s)=0$ using Wolfram Alpha\footnote{
    command:\texttt{ 
    derivative by s for s(d/s)\textasciicircum a(1-d/s)\textasciicircum b
    and determine its roots
    }
    }
    we obtain
    \begin{align}\label{eq:sx}
s = \frac{\beta(a+b-1)}{a-1}, \quad x = \frac{(a-1)^2}{\beta(a+b-1)^2}.
    \end{align}
    Since $s < 1$, 
    \eqref{eq:sx} provides an upper bound on $\beta$. Similarly, since $x \leq 1$, \eqref{eq:sx} gives a lower bound on $\beta$. These bounds are
    \[
\beta \in \left[ \frac{(a-1)^2}{(a+b-1)^2}, \frac{a-1}{(a+b-1)}  \right].
    \]
After substitution of $s$ from \eqref{eq:sx} to $f$ we obtain a linear estimate for the density of labeled $S_{a,b}$ as a function of $\beta$
\[
\ell(\beta) := \beta
 \frac{(a - 1)^{a-1}b^b}{(a-1 + b)^{a-1+b}}.
\]

(ii)
A clique of size $(1-s)n$ completely connected to a regular graph $R$ of size $sn$. This is the complement of (i), so the previous formulas give us the density for $S_{b,a}$ for edge density $1-\beta$. 
The density of $S_{a,b}$ as a function of $\beta$ in this case is
\[
\ell{}c(\beta) := (1-\beta)
 \frac{a^a(b-1)^{b-1}}{(a+ b-1)^{a+b-1}}
\quad\  \text{ for }\quad 
\beta \in \left[ 1-\frac{b-1}{(a+b-1)}, 1-\frac{(b-1)^2}{(a+b-1)^2} \right].
\]

(iii) For a regular graph of edge density $\beta$,  the density of labeled $S_{a,b}$ is
\[
r(\beta) := \beta^a (1-\beta)^b.
\]

(iv) A graph $G$ of edge density $\beta$ with the vertex set partitioned into parts $X,Y,Z$, where $G[X,Y]$ is complete bipartite, $G[Y]$ is a clique and there are no other edges. 
The optimal sizes of $X$, $Y$ and $Z$ and the resulting density of labeled $S_{a,b}$ come from  the following polynomial program.
\[
s(\beta) := 
\begin{cases}
\text{maximum} &  xy^a(1-y)^b + y(x+y)^a(1-x-y)^b  \\
\text{subject to} & x+y \leq 1 \\
& 2xy + y^2 = \beta \\
& x,y \geq 0.
\end{cases}
\]
For fixed values of $a$, $b$ and $\beta$ the program has one variable and can be solved using calculus (or more accurately, a computer algebra program). Unfortunately, the solution as a function of $\beta$ is a root of a polynomial expression depending on $a$ and $b$.
This construction has two special cases when $X=\emptyset$ or $Z =\emptyset$ that are easier to handle and seem to appear often.

(ivX)
If $X=\emptyset$ then the clique $Y$ has $\beta^{1/2}n$ vertices and the rest are isolated vertices in $Z$.
The density of labeled $S_{a,b}$ is
\[
c(\beta) := \beta^{1/2} \beta^{a/2}(1-\beta^{1/2})^b.
\]

(ivZ)
 If $Z=\emptyset$ then $Y$ is a set of $(1-(1-\beta)^{1/2})n$ universal vertices and $G[X]$ is an independent set. Notice that this is a complement of (ivX).
 The density of labeled $S_{a,b}$ is
\[
cc(\beta) := (1-\beta)^{1/2} (1-(1-\beta)^{1/2})^a (1-\beta)^{b/2}.
\]

(v)
The complement of construction (iv).
That means a graph $G$ of edge density $\beta$ with the vertex set partitioned into sets $X,Y,Z$ where $Z$ is the set of  universal vertices, $G[X]$ induces a clique and there are no other edges.

\[
cs(\beta) := 
\begin{cases}
\text{maximum} &  xy^b(1-y)^a + y(x+y)^b(1-x-y)^a  \\
\text{subject to} & x+y \leq 1 \\
& 2xy + y^2 = 1-\beta \\
& x,y \geq 0.
\end{cases}
\]

With these constructions, we state a conjecture of the profile for $S_{2,1}$.

\begin{conjecture}\label{conj:S21}
For every $S_{2,1}$-good $(G_n)$
\[
\left.
\begin{array}{ll}
cc(\beta) & \text{if } \beta \in [0,x] \\
c(\beta)  & \text{if } \beta \in [x,1]
\end{array}
\right\}
\leq
\rho(S_{2,1}, (G_n))
 \leq \begin{cases}  
 s(\beta) & \text{if } \beta \in [0,y] \\ 
 c(\beta) & \text{if } \beta \in [y,1/4] \\ 
 \ell(\beta) & \text{if } \beta \in [1/4, 1/2]\\
 r(\beta) & \text{if }  \beta \in [1/2,1],
 \end{cases}
\]
where $x \in (0,1)$ is  a solution to $cc(x) = c(x)$, i.e. $16x^3 -40x^2 +41x-16=0$
and $y \in (0,1/4)$ is a root of a degree 6 polynomial coming from determining the threshold when the optimal solution to $s(\beta)$ happens at $x=0$.
\end{conjecture}

The visualization of the conjecture is in Figure~\ref{fig:S21}. Notice that the upper bound in the conjecture describes $I(S_{2,1},\beta)$.
Since the number of labeled copies of $S_{2,1}$ in an $n$-vertex graph $G=(V,E)$ is upper bounded by 
$\sum_{v \in V} d_v^2\cdot(n-d_v)$,
Proposition~\ref{degree_sum} proves the upper bound from the conjecture for $\beta\ge 1/2$.

The method of the proof of Proposition~\ref{degree_sum} proves the conjecture for $\beta\ge 1/4.$ Indeed, the application of Lemma~\ref{opt} gives as an extremal construction a regular graph plus isolated vertices, which is the construction for $I(S_{2,1},\beta)$ when $\beta \geq 1/4$. 
More precisely, the application of Lemma~\ref{opt}
yields that the maximum for the sum of functions $f(x)=(1-x)x^2$ is achieved if there are $m$ values, each are equal to $\alpha$.
The sum is $m\alpha^2(1-\alpha)=\beta n\alpha(1-\alpha)$
since $m\alpha = \beta n$.\footnote{Here $\beta \in [0,1]$ while in the referenced proof $d \in [0,n]$.}
Since $\alpha(1-\alpha)$ is concave, maximized at $\alpha=1/2$ and $\beta \leq \alpha$, the maximum of the sum is achieved at $\alpha = \max\{1/2,\beta\}$. 
In order for the degree sequence of $m$ entries each of value $\alpha n$ to be realized, we need $m \ge \alpha n$, which is equivalent to the condition $\beta n \geq \alpha^2n$. 
Since $\alpha$ is $\max\{1/2,\beta\}$, the condition holds for $\beta\ge 1/4$.

We used flag algebras for calculating both upper and lower bounds at multiple fixed values of $\beta$ and the results were numerically matching  the conjecture. 
While ${\ell}c(\beta)$ and $cs(\beta)$ are not used in Conjecture~\ref{conj:S21}, we included them for completeness since they are useful for describing $I(S_{a,b},\beta)$ for other values of $a$ and $b$.

\begin{figure}[ht]
\begin{center}
\includegraphics[width=8cm]{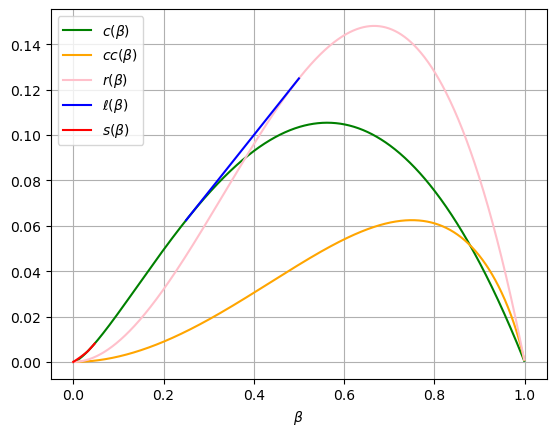}
\includegraphics[width=8cm]{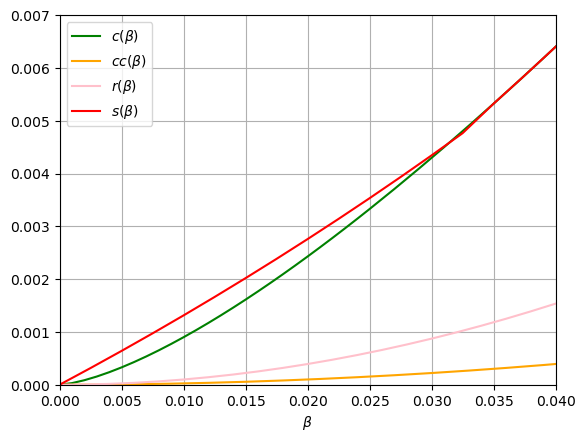}
\end{center}
    \caption{Functions used in Conjecture~\ref{conj:S21}.}
    \label{fig:S21}
\end{figure}

\section{Acknowledgment}
This material is based upon work supported by the National Science Foundation under Grant No. DMS-1928930, while the authors were in residence at the Simons Laufer Mathematical Sciences Institute in Berkeley, California, during the spring of 2025.
The computer calculations were performed mainly on the Alderaan cluster.
The Alderaan cluster was funded by the National Science Foundation Campus Computing program grant OAC-2019089 and a contribution from the Simons Foundation.
It is operated by the Center for Computational Mathematics, Department of Mathematical and Statistical Sciences, College of Liberal Arts and Sciences at the University of Colorado Denver.

\bibliographystyle{plainurl}
\bibliography{references.bib}

\begin{thebibliography}{10}

\bibitem{basit2025semiinducibilityproblem}
Abdul Basit, Bertille Granet, Daniel Horsley, André Kündgen, and Katherine
  Staden.
\newblock The semi-inducibility problem, 2025.
\newblock \href {https://arxiv.org/abs/2501.09842} {\path{arXiv:2501.09842}}.

\bibitem{bodnar2025}
Levente Bodnár and Oleg Pikhurko.
\newblock Semi-inducibility of 4-vertex graphs, 2025.
\newblock \href {https://arxiv.org/abs/2510.24336} {\path{arXiv:2510.24336}}.

\bibitem{bodnar2025exactinducibility}
Levente Bodnár and Oleg Pikhurko.
\newblock Some exact inducibility-type results for graphs via flag algebras,
  2025.
\newblock \href {https://arxiv.org/abs/2507.01596} {\path{arXiv:2507.01596}}.

\bibitem{chen2025maximizingalternatingpathsentropy}
Hao Chen, Felix~Christian Clemen, and Jonathan~A. Noel.
\newblock Maximizing alternating paths via entropy, 2025.
\newblock \href {https://arxiv.org/abs/2505.03903} {\path{arXiv:2505.03903}}.

\bibitem{chen2025}
Hao Chen and Jonathan~A. Noel.
\newblock On alternating 6-cycles in edge-coloured graphs, 2025.
\newblock \href {https://arxiv.org/abs/2505.09809} {\path{arXiv:2505.09809}}.

\bibitem{DHLMNPRS}
D.~Dellamonica, Jr., P.~Haxell, T.~{\L}uczak, D.~Mubayi, B.~Nagle, Y.~Person,
  V.~R\"odl, and M.~Schacht.
\newblock Tree-minimal graphs are almost regular.
\newblock {\em J. Comb.}, 3(1):49--62, 2012.
\newblock \href {https://doi.org/10.4310/JOC.2012.v3.n1.a2}
  {\path{doi:10.4310/JOC.2012.v3.n1.a2}}.

\bibitem{LMR}
Xizhi Liu, Dhruv Mubayi, and Christian Reiher.
\newblock The feasible region of induced graphs.
\newblock {\em J. Combin. Theory Ser. B}, 158:105--135, 2023.
\newblock \href {https://doi.org/10.1016/j.jctb.2022.09.003}
  {\path{doi:10.1016/j.jctb.2022.09.003}}.

\bibitem{Nikiforov2005}
Vladimir Nikiforov.
\newblock The minimum number of 4-cliques in graphs with triangle-free
  complement, 2005.
\newblock \href {https://arxiv.org/abs/math/0501211}
  {\path{arXiv:math/0501211}}.

\bibitem{Pippenger1975}
Nicholas Pippenger and Martin~Charles Golumbic.
\newblock The inducibility of graphs.
\newblock {\em Journal of Combinatorial Theory, Series B}, 19(3):189–203,
  December 1975.
\newblock \href {https://doi.org/10.1016/0095-8956(75)90084-2}
  {\path{doi:10.1016/0095-8956(75)90084-2}}.

\bibitem{Razborov}
Alexander~A. Razborov.
\newblock Flag algebras.
\newblock {\em J. Symbolic Logic}, 72(4):1239--1282, 2007.
\newblock \href {https://doi.org/10.2178/jsl/1203350785}
  {\path{doi:10.2178/jsl/1203350785}}.

\bibitem{ReiherStar}
Christian Reiher and Stephan Wagner.
\newblock Maximum star densities.
\newblock {\em Studia Sci. Math. Hungar.}, 55(2):238--259, 2018.
\newblock \href {https://doi.org/10.1556/012.2018.55.2.1395}
  {\path{doi:10.1556/012.2018.55.2.1395}}.

\end{thebibliography}

\appendix

\section{$P_{EENN}$}\label{app:PEENN}

This part provides the expansion of \eqref{eq:FAPEENN}.
Let $\mathcal{F}_5$ be the set of    
$5$-vertex graphs, up to isomorphism. 
We want to rewrite the right hand side as 
\[
\sum_{F \in \mathcal{F}_5} c_{a,F} F \leq \max_{F \in \mathcal{F}_5} c_{a,F} = 0. 
\]

The values of $c_{a,F} $ are the following polynomials in $a$, $B$ and $C$.

\begin{itemize}
\item[ ]
$\sum_{F \in \mathcal{F}_5} c_{a,F} F$ = 
 $0$ \Fe50{1 1 1 1 1 1 1 1 1 1}
 $+0$ \Fe50{1 1 1 1 1 1 1 1 1 2}
 $+0$ \Fe50{1 1 1 1 1 1 1 2 2 2}
 $+0$ \Fe50{1 1 1 1 2 2 2 2 2 2}
 $+0$ \Fe50{2 2 2 2 2 2 2 2 2 2}
\item[ ]
 $+\left(-12 \, a^{8} + 12 \, a^{7} + 4 \, a^{6} - 4 \, a^{5} + 12 \, a^{4} - 1.5 \, B a^{2} - 16 \, a^{3} + 4 \, a^{2}\right)$ \Fe50{1 1 1 1 1 1 1 1 2 2}
\item[ ] 
 $+\left(-36 \, a^{8} + 36 \, a^{7} + 12 \, a^{6} - 12 \, a^{5} + 18 \, a^{4} - 4.5 \, B a^{2} - 30 \, a^{3} + 12 \, a^{2}\right)$ \Fe50{1 1 1 1 1 1 2 1 2 2}
\item[ ]
 $+\left(-72 \, a^{8} + 72 \, a^{7} + 24 \, a^{6} - 24 \, a^{5} + 24 \, a^{4} - 9 \, B a^{2} + 6 \, C a^{2} - 48 \, a^{3} - 12 \, C a + 24 \, a^{2} + 6 \, C\right)$ \Fe50{1 1 1 2 1 1 2 1 2 2}
\item[ ]
 $+\left(4 \, a^{2} - 4 \, a\right)$ \Fe50{1 1 1 1 1 1 2 2 1 1}
\item[ ]
 $+\left(-24 \, a^{8} + 24 \, a^{7} + 8 \, a^{6} - 8 \, a^{5} + 12 \, a^{4} - 3 \, B a^{2} - 18 \, a^{3} + 10 \, a^{2} - 4 \, a\right)$ \Fe50{1 1 1 1 1 1 2 2 1 2}
\item[ ]
 $+\left(-24 \, a^{8} + 24 \, a^{7} + 8 \, a^{6} - 8 \, a^{5} + 12 \, a^{4} - 3 \, B a^{2} - 26 \, a^{3} + 18 \, a^{2} - 4 \, a\right)$ \Fe50{1 1 1 1 1 1 2 2 2 2}
\item[ ]
 $+\left(-12 \, a^{8} + 12 \, a^{7} + 16 \, a^{6} - 16 \, a^{5} + 2 \, a^{4} - 1.5 \, B a^{2} - 2 \, a^{3} + 4 \, a^{2} + 1.5 \, B - 4 \, a\right)$ \Fe50{1 1 1 2 1 1 2 2 1 1}
\item[ ]
 $+\left(-48 \, a^{8} + 48 \, a^{7} + 28 \, a^{6} - 28 \, a^{5} + 12 \, a^{4} - 6 \, B a^{2} - 20 \, a^{3} + 12 \, a^{2} + 1.5 \, B - 4 \, a\right)$ \Fe50{1 1 1 2 1 1 2 2 1 2}
\item[ ]
 $+\left(-60 \, a^{8} + 60 \, a^{7} + 32 \, a^{6} - 32 \, a^{5} + 16 \, a^{4} - 7.5 \, B a^{2} + C a^{2} - 36 \, a^{3} - 2 \, C a + 24 \, a^{2} + 1.5 \, B + C - 4 \, a\right)$ \Fe50{1 1 1 2 1 1 2 2 2 2}
\item[ ]
 $+\left(-48 \, a^{8} + 48 \, a^{7} + 16 \, a^{6} - 16 \, a^{5} + 16 \, a^{4} - 6 \, B a^{2} - 24 \, a^{3} + 16 \, a^{2} - 8 \, a\right)$ \Fe50{1 1 1 1 1 2 2 2 2 1}
\item[ ]
 $+\left(-24 \, a^{8} + 24 \, a^{7} + 8 \, a^{6} - 8 \, a^{5} + 12 \, a^{4} - 3 \, B a^{2} - 28 \, a^{3} + 22 \, a^{2} - 6 \, a\right)$ \Fe50{1 1 1 1 1 2 2 2 2 2}
\item[ ]
 $+\left(-36 \, a^{8} + 36 \, a^{7} + 36 \, a^{6} - 36 \, a^{5} + 2 \, a^{4} - 4.5 \, B a^{2} - 4 \, a^{3} + 6 \, a^{2} + 3 \, B - 4 \, a\right)$ \Fe50{1 1 1 2 1 2 1 2 2 1}
\item[ ]
 $+\left(-48 \, a^{8} + 48 \, a^{7} + 40 \, a^{6} - 40 \, a^{5} + 8 \, a^{4} - 6 \, B a^{2} - 20 \, a^{3} + 16 \, a^{2} + 3 \, B - 4 \, a\right)$ \Fe50{1 1 1 2 1 2 1 2 2 2}
\item[ ]
 $+\left(-72 \, a^{8} + 72 \, a^{7} + 60 \, a^{6} - 60 \, a^{5} + 6 \, a^{4} - 9 \, B a^{2} - 10 \, a^{3} + 8 \, a^{2} + 4.5 \, B - 4 \, a\right)$ \Fe50{1 1 1 2 1 2 2 2 2 1}
\item[ ]
 $+\left(-60 \, a^{8} + 60 \, a^{7} + 56 \, a^{6} - 56 \, a^{5} + 8 \, a^{4} - 7.5 \, B a^{2} + C a^{2} - 24 \, a^{3} - C a + 20 \, a^{2} + 4.5 \, B - 4 \, a\right)$ \Fe50{1 1 1 2 1 2 2 2 2 2}
\item[ ]
 $+\left(-108 \, a^{8} + 108 \, a^{7} + 108 \, a^{6} - 108 \, a^{5} - 13.5 \, B a^{2} + 9 \, B\right)$ \Fe50{1 1 2 2 1 2 2 2 2 1}
\item[ ]
 $+\left(-72 \, a^{8} + 72 \, a^{7} + 96 \, a^{6} - 96 \, a^{5} - 9 \, B a^{2} + 6 \, C a^{2} - 12 \, a^{3} - 6 \, C a + 12 \, a^{2} + 9 \, B\right)$ \Fe50{1 1 2 2 1 2 2 2 2 2}
\item[ ]
 $+\left(6 \, a^{4} - 6 \, a^{3} + 6 \, a^{2} - 6 \, a\right)$ \Fe50{1 1 1 2 2 2 1 2 1 1}
\item[ ]
 $+\left(-36 \, a^{8} + 36 \, a^{7} + 24 \, a^{6} - 24 \, a^{5} + 8 \, a^{4} - 4.5 \, B a^{2} - 12 \, a^{3} + 10 \, a^{2} + 1.5 \, B - 6 \, a\right)$ \Fe50{1 1 1 2 2 2 1 2 1 2}
\item[ ]
 $+\left(-48 \, a^{8} + 48 \, a^{7} + 40 \, a^{6} - 40 \, a^{5} + 8 \, a^{4} - 6 \, B a^{2} - 16 \, a^{3} + 14 \, a^{2} + 3 \, B - 6 \, a\right)$ \Fe50{1 1 1 2 2 2 1 2 2 2}
\item[ ]
 $+\left(-36 \, a^{8} + 36 \, a^{7} + 48 \, a^{6} - 48 \, a^{5} + 6 \, a^{4} - 4.5 \, B a^{2} - 18 \, a^{3} + 18 \, a^{2} + 4.5 \, B - 6 \, a\right)$ \Fe50{1 1 1 2 2 2 2 2 2 2}
\item[ ]
 $+\left(-48 \, a^{8} + 48 \, a^{7} + 16 \, a^{6} - 16 \, a^{5} + 16 \, a^{4} - 6 \, B a^{2} - 24 \, a^{3} + 16 \, a^{2} - 8 \, a\right)$ \Fe50{1 1 2 2 2 1 2 1 2 2}
\item[ ]
 $+\left(-60 \, a^{8} + 60 \, a^{7} + 80 \, a^{6} - 80 \, a^{5} - 10 \, a^{4} - 7.5 \, B a^{2} + 10 \, a^{3} + 7.5 \, B\right)$ \Fe50{1 1 2 2 2 1 2 2 1 1}
\item[ ]
 $+\left(-72 \, a^{8} + 72 \, a^{7} + 72 \, a^{6} - 72 \, a^{5} - 9 \, B a^{2} - 2 \, a^{3} + 4 \, a^{2} + 6 \, B - 2 \, a\right)$ \Fe50{1 1 2 2 2 1 2 2 1 2}
\item[ ]
 $+\left(-60 \, a^{8} + 60 \, a^{7} + 56 \, a^{6} - 56 \, a^{5} + 6 \, a^{4} - 7.5 \, B a^{2} - 12 \, a^{3} + 10 \, a^{2} + 4.5 \, B - 4 \, a\right)$ \Fe50{1 1 2 2 2 1 2 2 2 2}
\item[ ]
 $+\left(-84 \, a^{8} + 84 \, a^{7} + 88 \, a^{6} - 88 \, a^{5} - 2 \, a^{4} - 10.5 \, B a^{2} + 2 \, a^{3} + 7.5 \, B\right)$ \Fe50{1 1 2 2 2 2 2 2 2 1}
\item[ ]
 $+\left(-48 \, a^{8} + 48 \, a^{7} + 64 \, a^{6} - 64 \, a^{5} + 2 \, a^{4} - 6 \, B a^{2} - 6 \, a^{3} + 6 \, a^{2} + 6 \, B - 2 \, a\right)$ \Fe50{1 1 2 2 2 2 2 2 2 2}
\item[ ]
 $+\left(-72 \, a^{8} + 72 \, a^{7} + 72 \, a^{6} - 72 \, a^{5} - 9 \, B a^{2} + 2 \, C a^{2} + 6 \, B\right)$ \Fe50{1 2 2 2 2 2 2 1 2 2}
\item[ ]
 $+\left(-36 \, a^{8} + 36 \, a^{7} + 48 \, a^{6} - 48 \, a^{5} - 4.5 \, B a^{2} + 3 \, C a^{2} + 4.5 \, B\right)$ \Fe50{1 2 2 2 2 2 2 2 2 2}
\end{itemize}

Plots of non-zero polynomials $c_{a,F}$ are depicted in Figure~\ref{fig:PEENNbluered} for both options of values of $B$ and $C$.
\begin{figure}
\begin{center}
\includegraphics[width=8cm]{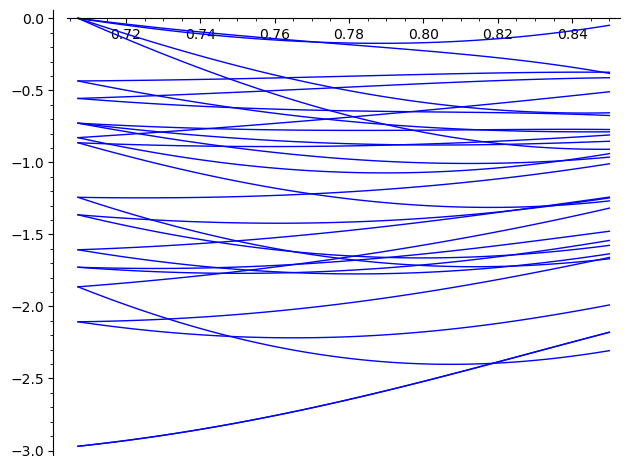}
\includegraphics[width=8cm]{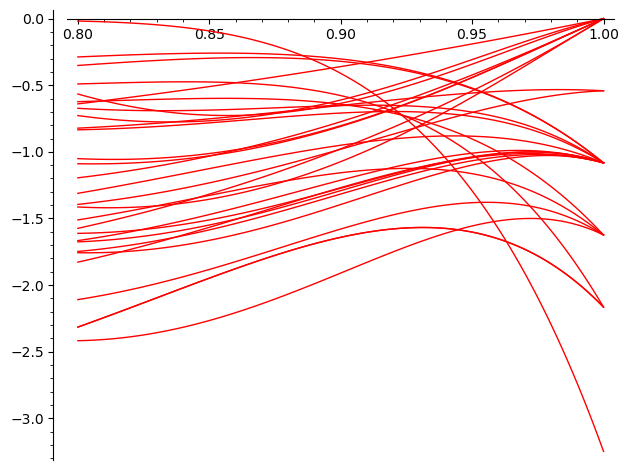}
\end{center}
    \caption{Case $B=C=\sqrt{2}-1$ on the left and on the right $B=0.361$ and $C=0$.} 
    \label{fig:PEENNbluered}
\end{figure}

\end{document}